\documentclass[12pt]{amsart}

\usepackage{amsmath,amssymb,amsthm,xypic,epsfig }
\usepackage{enumerate}
\usepackage[dvips]{color}
\usepackage{version}

\DeclareFontFamily{OT1}{rsfs}{}
\DeclareFontShape{OT1}{rsfs}{n}{it}{<-> rsfs10}{}
\DeclareMathAlphabet{\curly}{OT1}{rsfs}{n}{it}

\newtheorem{Thm}{Theorem}[section]
\newtheorem{Lem}[Thm]{Lemma}

\newtheorem{Cor}[Thm]{Corollary}

\newtheorem{Prop}[Thm]{Proposition}
\newtheorem{DefProp}[Thm]{Definition-Proposition}
\newtheorem{Conj}[Thm]{Conjecture}
\newtheorem{``Conj"}[Thm]{``Conjecture"}

\theoremstyle{remark}
\newtheorem{Rem}[Thm]{Remark}
\newtheorem{Ex}[Thm]{Example}
\theoremstyle{definition}
\newtheorem{Def}[Thm]{Definition}

\excludeversion{memo}

\begin{document}

\title[Modular heights for general varieties]{Canonical K\"ahler metrics and arithmetics  \\  
-generalizing Faltings heights - } 

\author{Yuji Odaka}

\date{submitted: 4th October, 2015.  \\ revised: 15th, September 2016. }

\maketitle
\thispagestyle{empty}

\begin{abstract}
We extend the Faltings modular heights of abelian varieties to general 
arithmetic varieties 
and show direct relations with the K\"ahler-Einstein geometry, the Minimal Model 
Program, Bost-Zhang's heights and give some applications. 

Along the way, we propose ``arithmetic Yau-Tian-Donaldson conjecture'' 
(equivalence of a purely \textit{arithmetic} property of variety and its \textit{metrical} property) 
and partially confirm it. 
\end{abstract}

\tableofcontents


\section{Introduction}\label{intro}

\textit{To metrize (schemes or bundles on them) is to ``compactify''} (*)\\ 
- as a motto, this expresses the philosophy of S.~Arakelov (cf., \cite{Ara74}) whose aim was 
to get a nice \textit{intersection theory} for arithmetic varieties. 
G.~Faltings' proof of the Mordell conjecture \cite{Fal83} takes advantage 
of the intersection theory. In particular, he introduced the key invariant - 
\textit{Faltings (modular) height} of abelian varieties. 

Our general aim is, to show how this motto fairly compatibly fits to 
the recent studies of 
``canonical K\"ahler metrics'' such as K\"ahler-Einstein metrics and geometric flows towards them. 
In particular, we discuss the so-called Yau-Tian-Donaldson correspondence 
- the equivalence of the existence of 
\textit{canonical K\"ahler metric} and some variant of GIT stability. 
In such studies related to the canonical K\"ahler metrics, 
we have occasionally encountered the problem of constructing compactification of moduli spaces 
(``\textit{K-moduli}''\footnote{i.e. vaguely speaking, moduli of K-stable (polarized) varieties \\ 
cf., \cite{FS90},\cite{Fuj90},\cite{MM90},\cite{Don97},\cite{Oda10},\cite{Oda11a},\cite{Spo12},\cite{DS14},\cite{OSS12}, and the last year's settlement of (smoothable) Fano case 
\cite{SSY14},\cite{LWX14},\cite{Oda15b}. For general K-moduli 
conjecture cf., \cite[section 3]{Oda12} (also last sub-section \ref{a.K.mod} of this paper.) }). 
Thus the author suspects that the above sentence $(*)$ expressing 
the Arakelov's philosophy sounds somehow familiar to experts of the recent studies of canonical 
K\"ahler metrics and related (K-)moduli theory. 
In other words, it also vaguely gives yet another heuristic ``explanation'' 
why the compactification problem naturally arises in such studies of metrics. 

The basic invariant in the field is the \textit{Donaldson-Futaki invariants} (\cite{Tia97}, \cite{Don02}). 
It captures asymptotic of the \textit{K-energy} (\cite{Mab86})
of K\"ahler metrics along certain variations. 

In this paper, 
we unify all the above invariants which originally appeared in different fields 
- Faltings height, Donaldson-Futaki invariant and K-energy - into 
one, as \textbf{Arakelov-Donaldson-Futaki invariant} 
 $h_{K}(\mathcal{X},\mathcal{L},h) (\in \mathbb{R})$ 
for arithmetic polarized variety. We also call it simply 
\textbf{modular height} for brevity. If it would sound too brave or confusing, 
``K-modular height'' would be a better alternative name in some contexts. 
We also introduce its siblings invariants and give some applications.

We collect some basic properties of the (K-)modular height below, 
deliberately stated in a vague form for simpler illustration. 
The precise meanings are put in later sections. 

\begin{Thm}\label{Intro.thm}
The invariant $h_{K}(\mathcal{X},\mathcal{L},h)\in \mathbb{R}$ 
(we introduce at \ref{aDF}) for an 
arithmetic polarized projective scheme $(\mathcal{X},\mathcal{L})$ over $\mathcal{O}_{K}$, 
the ring of integers in a number field $K$, 
\footnote{
The subscript ``$K$'' of $h_{K}$ 
comes from ``K''-stability rather than the base field. Indeed, the quantity $h_{K}$ remains the 
same after extension of the scalars. }
attached with a hermitian metric 
$h$ on the complex line bundle $\mathcal{L}(\mathbb{C})$ over 
$\mathcal{X}(\mathbb{C})$ which is invariant under complex conjugate, satisfies the followings. 
We denote the generic fiber of $(\mathcal{X},\mathcal{L})$ as $(X,L)$. 
\begin{enumerate}

\item 
If $(\mathcal{X},\mathcal{L},h)$ is a polarized abelian scheme 
with the cubic metric $h$, 
$h_{K}(\mathcal{X},\mathcal{L},h)$ essentially coincides with the \textbf{Faltings height} \cite{Fal83} 
of the generic fiber $X$. 
Please see Theorem \ref{Falt.ht} for the details, where we allow bad reductions. 
We remark that as a special case of (\ref{Main.Thm(5)}) and \ref{hk.min} later, 
we see that such $(\mathcal{X},\mathcal{L},h)$ minimizes $h_{K}$ among all metrized integral models. 

\item With respect to change of metrics, it behaves as 
$$h_{\it K}(\mathcal{X},\mathcal{L},h\cdot e^{-2\varphi})-
h_{\it K}(\mathcal{X},\mathcal{L},h)
=\frac{(L^{{\it dim}(X)})}{[K:\mathbb{Q}]}\cdot \mu_{\omega_{h}}(\varphi),$$
where $\mu$ denotes the \textbf{Mabuchi K-energy} \cite{Mab86}. 

\item If we birationally 
change the model $(\mathcal{X},\mathcal{L})$ along some finite closed fibers (while 
preserving $h$), $h_{K}(\mathcal{X},\mathcal{L},h)$ behaves very similarly to 
the \textbf{Donaldson-Futaki invariant} in equi-characteristic situation 
(cf., \cite{Don02},\cite{Wan12},\cite{Oda13} for its definition and formulae). 
Please find the precise meaning in the section \ref{sec.2} later. 

\item From (2),(3), it follows that $h_{K}$ \textbf{decreases} along a combination of 
\textbf{arithmetic MMP with scaling} and the \textbf{K\"ahler-Ricci flow} (which are compatible). 
Please find the precise meaning at \ref{hK.decrease}. 

\item \label{Main.Thm(5)}
Given a polarized projective variety $(X,L)$ defined over $K$ 
which possesses K\"ahler-Einstein metrics (over infinite places), 
$h_{K}$ of integral models of $(X,L)$ \textbf{minimizes at ``minimal-like'' integral model}
over $\mathcal{O}_{K}$ (defined in terms of birational geometry) \textbf{with the K\"ahler-Einstein metric}  
attached. Please find the precise meaning at \ref{hk.min}.

\end{enumerate}
\end{Thm}

\noindent 
The invariant $h_{K}$ can be also seen as ``limited version'' of some normalisation of 
the \textbf{Bost-Zhang's height} 
(\cite{Bost94}, \cite{Bost96}, \cite{Zha96}) as we explain later. 

Recall that \cite{Tia97},\cite{Don02} introduced the Donaldson-Futaki invariant 
for ``test configuration'' which is a flat isotrivial family (with $\mathbb{C}^{*}$-action) 
over $\mathbb{C}$ or complex disk, and 
\cite{Wan12}, \cite{Oda13} independently showed that it 
is an intersection number on the global total space, as a simple application of the 
(equivariant) Riemann-Roch type theorems. (1)(2)(3) explain and extend some of the known facts in the field of 
canonical K\"ahler metrics (as well as the author's study of K-stability) 
while some variant of (2)(3)(4) (see section \ref{MA}) 
explain and refine some theorems in Arakelov geometry 
by \cite{Bost94},\cite{Bost96},\cite{Zha96} etc. 

Indeed, as we show in section \ref{sec.3} via some 
asymptotic analysis of \textit{Ray-Singer torsion} \cite{RS73}, 
our modular height $h_{K}(\mathcal{X},\mathcal{L})$ is exactly what controlls 
the first non-trivial 
asymptotic behaviour of the (Chow) heights of Bost-Zhang (\cite{Bost94}, \cite{Bost96}, \cite{Zha96}) 
of $\mathcal{X}$ embedded by $|{\bar{\mathcal{L}}}^{\otimes m}|$ 
with respect to $m\to \infty$. 
This is yet another important feature of our modular height $h_{K}$. 

Thus, in particular, via our ``unification'' $h_{K}$ with its properties proved in the present paper, 
direct relations among the three below (which appeared in different contexts) follows indirectly. 
\begin{enumerate}
\renewcommand{\labelenumi}{\alph{enumi}).}
\item Faltings height \cite{Fal83} of arithmetic abelian varieties, 
\item K-energy \cite{Mab86}, Donaldson-Futaki invariant \cite{Don02} 
\item Bost-Zhang's (\cite{Bost94},\cite{Bost96},\cite{Zha96}) heights. 
\end{enumerate}

That is, we have the following relations: 

\begin{Cor}

Among the above three invariants, 

\begin{itemize}

\item ((a)$\leftrightarrow$(b)) \\ 
``\textit{Faltings height for abelian varieties is essentially 
(a special case of) arithmetic version of Donaldson-Futaki invariant}''. 

\item ((b)$\leftrightarrow$ (c)) \\ 
``\textit{Mabuchi's K-energy is essentially an infinite place part of a limit of 
(modified) Bost-Zhang's heights}''. 
\item ((c)$\leftrightarrow$ (a)) \\ 
``\textit{Faltings height for abelian varieties is essentially a limit of (modified) Bost-Zhang's 
heights}''. 
\end{itemize}
\end{Cor}

We again deliberately gave rough statements above, rather than lengthy precise statements, 
as both the precise statements and proofs 
will be clear to the readers of details of section \ref{sec.2} and 
\ref{sec.3}. The first relation ((a)$\leftrightarrow$(b)) follows from 
(\ref{Falt.ht}), combined with (\ref{hk.min}), (\ref{ss..red}). 
The second relation ((b)$\leftrightarrow$ (c)) follows from 
(\ref{h.lim.hk}) (combined with (\ref{ADF.K})). The last relation 
((c)$\leftrightarrow$ (a)) follows from 
(\ref{Falt.ht}) combined with (\ref{h.lim.hk}). 

For those who are not really tempted to go through the details, 
we make brief comments about essential points of the above. 
The essential point of the first relation is (coincidence of scheme theoretic line bundle and) calculation of \textit{Weil-Petersson metric} on the base. 
The second relation can be seen as a sort of quantisation of K-energy. The last relation 
essentially follows from an asymptotic analysis of 
\textit{Ray-Singer torsion} combined with a refinement of asymptotic Hilbert-Samuel formula \ref{aHS}. 

We remark that Theorem \ref{Intro.thm}(2) gives a yet another way of connecting two of $(b)$, i.e. 
Mabuchi K-energy and Donaldson-Futaki invariant, via our $h_{K}$ (\ref{ADF.K}) and roughly shows us that 
``\textit{Mabuchi's K-energy is essentially an intersection number}''.

In section \ref{sec.2}, several other 
basic ``height type'' invariants are introduced and studied as well partially for future further analysis. 
They are connected simultaneously to, on one hand, 
some other functionals over the space of K\"ahler metrics (cf., e.g., \cite{BBGZ13}) and on the other hand, some other intersection theoretic quantities (analysed in \cite{Oda11a},\cite{Der14},\cite{BHJ15a} etc). 

One point of this whole set of analogies is that the 
two protagonists in the field, metric and polarization 
(i.e., ample line bundle), 
both need to be 
``positive'' by their definitions (or, by its natures) 
and the varieties or families are the ``models'' which ``realize'' the required 
positivity. For that realisation, we need change of models by geometric 
flows of metrics or equivalently, 
(mostly-)birational modifications, such as the Minimal Model Program. 
Then such flows are supposed to arrive at nice canonical metrics or models, which 
gives ``canonical" compactification of moduli spaces. 
At least, in this way, we can give a yet another heuristic explanation to 
the Yau-Tian-Donaldson conjecture, the K-moduli conjecture (cf., \cite[Conjecture 3.1]{Oda12}), 
and perhaps also the (arithmetic) Minimal Model Program itself. 

Most parts of this paper have their origins in algebro-geometric versions 
(or K\"ahler geometric versions) which are already established, 
and thus corresponding geometric 
papers are cited at each appropriate sections. We recommend interested readers 
to review the geometric counterparts. Another recent extension of the Donaldson-Futaki invariant or 
the K-stability theory in a still algebro-geometric realm (such as 
partial resolutions of singularities) is in \cite{Oda15a}. 

Finally we would like to mention that there is recently another interesting 
generalisation to different direction of 
the Faltings heights i.e. introduced for \textit{motives} (``height of motives") 
and its studies due to K.~Kato and T.~Koshikawa \cite{Kat14}, \cite{Kos15}. 

We organize our paper as follows. 
After this introductory section, in the section \ref{sec.2}, we introduce our modular height $h_{K}$ 
and its related variants, mainly after \cite{Oda11a} and \cite{BHJ15a}. Then in section \ref{sec.3},  
we discuss applcations as well as deeper results and speculations. 
First, we show that ``asymptotic Chow semistability'' does \textit{not} 
admit semistable reduction in that sense. 
Second, we propose an arithmetic version of the Yau-Tian-Donaldson conjecture. 
During the arguments, by using a result about asymptotic behaviour of Ray-Singer analytic torsion due to  
Bismut-Vasserot \cite{BV89}, we show that ``quantisation'' of our invariants is essentially the original 
heights introduced by Bost, Zhang in \cite{Bost94},\cite{Bost96},\cite{Zha96}. 

In this paper, unless otherwise mentioned, we work under the following setting.

\subsection{Notations and Conventions}\label{Not.Cov}
\begin{enumerate}

\item
We work with arithmetic scheme of the form 
$\pi\colon \mathcal{X} \to {\it Spec}(\mathcal{O}_{K})=:C$ with relative dimension $n$ 
where $K$ is 
a number field and $\mathcal{O}_{K}$ is the ring of integers of $K$, 
$\mathcal{X}$ is flat and (relatively) projective over $C$ and 
$\mathcal{L}$ is a (relatively) ample line bundle on $\mathcal{X}$. The generic point of $C$ is 
denoted by $\eta$ and the generic fiber is denoted as $(\mathcal{X}_{\eta},\mathcal{L}_{\eta})$ 
or simply $(X,L)$. 

\item 
For simplicity, throughout this paper, we assume that $\mathcal{X}$ is 
normal. 
We put 
$K_{\mathcal{X}^{\it sm}/C}:=\wedge^{n}_{\mathcal{O}_{\mathcal{X}^{\it sm}}} 
\Omega_{\mathcal{X}^{\it sm}/\mathcal{O}_{K}}$ where $\mathcal{X}^{\it sm}\subset 
\mathcal{X}$ denotes the open dense subset of $\mathcal{X}$ where $\pi$ is smooth. 
Then we further assume, for simplicity, the ``$\mathbb{Q}$-Gorenstein condition'' i.e. 
with some $m\in \mathbb{Z}_{>0}$, $(K_{\mathcal{X}^{\it sm}/C})^{\otimes m}$ extends to 
an invertible sheaf 
$((K_{\mathcal{X}^{\it sm}/C})^{\otimes m})^{\widehat{}}$ 
on whole $\mathcal{X}$ 
(we call the condition ``$\mathbb{Q}$-Gorenstein" following the custom in 
birational algebraic geometry.) 
Then, the \textit{discrepancy} of $\mathcal{X}$ along some exceptional divisor over it is defined completely similarly as 
the geometric case (cf., \cite{KM98}) as follows. 

Suppose $f\colon \mathcal{Y}\to \mathcal{X}$ is a blow up morphism  
of arithmetic scheme $\mathcal{X}$ as above, where we also 
assume $\mathcal{Y}$ is normal. Then 
if we write 
$(K_{\mathcal{Y}})^{\otimes m}\otimes 
\pi^{*}
((K_{\mathcal{X}^{\it sm}/C})^{\otimes -m})^{\widehat{}}
=\mathcal{O}_{\mathcal{Y}}(\sum a_{i}mE_{i})$ with 
exceptional prime divisors $E_{i}$ of $\mathcal{Y}$ and 
$a_{i}\in \mathbb{Q}$, $a_{i}$ is called the discrepancy of 
$E_{i}$ over $\mathcal{X}$. 
Accordingly we call 
$\mathcal{X}$ is \textit{log-canonical} 
(resp., \textit{log-terminal}) 
when $a_{i}\ge -1$ (resp., $a_{i}> -1$) for all $f$ and $i$. 
Note that, instead of using resolutions of singularities, we use all 
normal blow ups. For the study of related classes of singularities from 
the discrepancy viewpoint, we recommend recent references 
e.g. \cite{Kol13}, \cite{Tan}. 

\item 
For such polarized arithmetic variety $(\mathcal{X},\mathcal{L})$, 
we associate complex geometric generic fiber $(\mathcal{X}(\mathbb{C})=:X_{\infty},\mathcal{L}
(\mathbb{C})=:L_{\infty})$. Note that $X_{\infty}$ is usually 
\textit{not} connected (though equidimensional) e.g. 
when the base field $K$ is not $\mathbb{Q}$ but it would not cause any technical problem. 

\item 
$h$ is a continuous hermitian metric of real type on $L_{\infty}$ which is $C^{\infty}$ at the smooth locus 
$X^{\it sm}_{\infty}$ of $X_{\infty}$, and $c_{1}(L_{\infty}|_{X_{\infty}^{\it sm}},h)$ extends as 
a closed 
\textit{positive} $(1,1)$ current with locally continuous potential (through singularities of 
$X_{\infty}$). 
In this paper, we call such a metric \textit{almost smooth} 
hermitian metric (of real type). 
The curvature of $h$ is assumed to be positive semi-definite. 
For $n+1$ such metrized line bundles $\bar{\mathcal{L}_{i}} (i=0,\cdots,n)$, 
via generic resolution (cf., e.g., \cite[5.1.1]{Mor14})), 
Gillet-Soul\'e intersection number \cite{GS90} is well-defined as we explain 
in the subsection \ref{subsec.2.1}. 

\item 
$c_{1}(L,h)$ means the first chern form (current) $c_{1}(L_{\infty},h)$, i.e. is locally 
$-\frac{i}{2\pi}\partial\bar{\partial}{\it log}(h(s,\bar{s}))$ where $s$ is arbitrary local 
non-vanishing holomorphic 
section of $L_{\infty}$. More precisely, it is a pushforward of the smooth (usual) first chern form at the 
generic resolution. We also denote it as $\omega_{h}$. 

\item 
$\mathcal{H}(L)$ means the space of appropriate hermitian metrics 
$\{\text{almost smooth $h$ on $L_{\infty}$ of real type with positive $c_{1}(L_{\infty},h)$}\}.$ 
In \textit{this} paper, it is enough to treat this only as a set. 

\item We denote the model with metric as $\pi\colon (\mathcal{X},\bar{\mathcal{L}}^{h}:=(\mathcal{L},h))\to 
{\it Spec}(\mathcal{O}_{K})$. We also sometimes write $\bar{\mathcal{L}}^{\omega_{h}}$ instead of 
$\bar{\mathcal{L}}^{h}$, where $\omega_{h}$ means 
$c_{1}(\mathcal{L}(\mathbb{C}),h)$ 
as above. 

\item Occasionally we fix a reference global model with the specified generic fiber, 
denoted as $\pi\colon (\mathcal{X}_{\it ref},\mathcal{L}_{\it ref}, h_{\it ref})\to 
{\it Spec}(\mathcal{O}_{K})$, and work with various $(\mathcal{X},\mathcal{L},h)$ with the 
isomorphic fiber (with possibly different metric) with specified isomorphisms. 

\item 
From the non-archimedean geometric perspective, it can be seen that 
main body of 
this paper discusses basically only model metrics. However, 
all the arguments in subsections \S \ref{subsec.2.1} to \S \ref{Ar.Aubin} and some other parts 
can be straightforwardly extended to that for semipositive \textit{adelic} 
(metrized) line bundles $\bar{\mathcal{L}}$ without any technical difficulties. To extend each claim in \S 2.1 to \S 2.7, 
where we assume 
vertical ampleness (resp., vertical nefness) of some 
arithmetic line bundles (model metrics), we can straightforwardly 
extend the claim to that for 
``vertically ample (resp., vertically nef)'' adelic (metrized) 
line bundles. Such vertical ampleness (resp., vertical nefness) 
of adelic (metrized) line bundle is simply defined as being 
a uniform limit 
\footnote{in the sense of \cite{Zha95b}} 
of vertically ample (resp., vertically nef) 
arithmetic line bundles (model metrics). 
Hence, we wish to just omit 
and possibly 
rewrite that extensions more explicitly in future. 
\footnote{Indeed, the only nontrivial technically necessary change in 
such extension of subsections \ref{subsec.2.1} to \ref{Ar.Aubin} is to replace 
the Moriwaki Hodge index theorem \cite{Mor96} for model metrics by 
the recent extension by Yuan-Zhang \cite[1.3]{YZ13} 
or the continuity of intersection numbers \cite[1.4 (a)]{Zha95b}. }

\end{enumerate}

For the above $\pi\colon (\mathcal{X},\bar{\mathcal{L}}=(\mathcal{L},h))\to {\it Spec}(\mathcal{O}_{K})$, 
we associate a real number invariant $h_{K}(\mathcal{X},\mathcal{L},h)$ 
which we call the Arakelov-Donaldson-Futaki invariant or (K-)modular height. It extends the 
Faltings' modular height of polarized abelian variety (as generic fiber) and encodes K-energy and 
Donaldson-Futaki invariant. For the precise meanings, please read below carefully. 

\subsection*{Acknowledgements} 
The author is grateful to R.~Berman, 
S.~Boucksom, H.~Guenancia, T.~Hisamoto, S.~Kawaguchi, 
T.~Koshikawa, S.~Mori, 
A.~Moriwaki, H.~Tanaka, T.~Ueda, K.~Yamaki, and K-I.~Yoshikawa and S-W.Zhang 
for helpful discussions and  encouragements. 
The author is partially supported by 
Grant-in-Aid for Young Scientists (B) 26870316 from JSPS.


\section{Arakelov intersection-theoretic functionals}\label{sec.2}

\subsection{General preparations on intersection theory}\label{subsec.2.1}

\subsubsection{Arakelov intersection numbers for singular varieties}

As in the notation set above in the introduction section, we work with 
$n+1$-dimensional arithmetic scheme of the form 
$\pi\colon (\mathcal{X},\mathcal{L})\to {\it Spec}(\mathcal{O}_{K})=:C$. 
Although the Arakelov-Gillet-Soul\'e's intersection theory (\cite{GS90}) is (usually) defined 
for regular scheme, as far as one only concerns about 
the intersection numbers of $n+1$ arithmetic line bundles, 
the full regularity is not required as we see below. 
One way to see it is via the use of 
generic resolution (cf., e.g., \cite[5.1.1]{Mor14}) as follows. 

\begin{DefProp}
Suppose $\mathcal{X}$ is a normal scheme which is flat and projective over 
$C={\it Spec}(\mathcal{O}_{K})$ and $n+1$ line bundles $\mathcal{L}_{0},\cdots,\mathcal{L}_{n}$ 
of $\mathcal{X}$ 
attached with almost smooth metrics $h_{i}$ (cf., Notation and Convention \ref{Not.Cov}), 
which we denote by $\bar{\mathcal{L}_{i}}=(\mathcal{L}_{i},h_{i})$ $(i=0,\cdots,n)$. 

If we take a birational proper morphism $\pi\colon \tilde{\mathcal{X}}\to \mathcal{X}$ 
from generically smooth $\tilde{\mathcal{X}}$ (it exists by Hironaka \cite{Hir64} 
cf., e.g., \cite[5.1.1]{Mor14}), 
then the definition of Gillet-Soul\'e's (cf., \cite{GS90})  
intersection number $(\pi^{*}\bar{\mathcal{L}}_{0}.\cdots.\pi^{*}\bar{\mathcal{L}}_{n})$ works 
and also does not depend on $\pi$. We simply denote the value by 
$(\bar{\mathcal{L}}_{0}.\cdots.\bar{\mathcal{L}}_{n})$. 
\end{DefProp}

To see the above well-definedness i.e. that the Gillet-Soul\'e intersection theory (cf., \cite{GS90}) works, 
the essential point is to confirm well-definedness of wedge product 
$$T\mapsto c_{1}(L,h)\wedge T$$ with any closed positive $(d,d)$-current $T$ $(0\le d< n)$ 
(the other term of $*$-product, i.e. $[{\it log}(h(s,s))]|_{Z(\mathbb{C})}$ for $Z\subset X$, 
is well-defined by local integrability of ${\it log}(h(s,s))$ 
with holomorphic section $s$). This desired wedge product is standard after Bedford-Taylor 
\cite[p4]{BT82} (cf., also \cite[III, section 3]{Dem}), as we are able to put 
$c_{1}(L,h)\wedge T=dd^{c}({\it log}(h(s,s))T)$ with non-vanishing local holomorphic section of 
$L$. 

As the proof of the independence from $\pi$ 
is straightforward by the use of common generic resolution 
(which again exists by \cite{Hir64}) and 
projection formula (cf., e.g., \cite[Proposition 5.5]{Mor14}), we omit the details of the proof. 
We denote the above intersection number simply as $(\bar{\mathcal{L}}_{0}.\cdots.\bar{\mathcal{L}}_{n})$ 
and will use it throughout this work.

\subsubsection{Change of metrics}

We introduce the Arakelov theoretic versions of the functionals of the space of K\"ahler metrics 
$\mathcal{H}(L)$ and show the 
compatibility with both the 
Non-archimedean analogues for test configurations in the style of \cite{BHJ15a} (which is 
for equicharacteristic base), 
and the classical K\"ahler version of the functionals such as \cite{Mab86}. 
In particular, our results give another explanation (in addition to \cite{BHJ15b}), 
with certain mathematical statements, why \cite{BHJ15a}'s intersection numbers type invariants 
can be seen as non-archimedean analogues of the corresponding functionals 
over the space of K\"ahler metrics. 

A point is that the Arakelov intersection theory can be decomposed into ``local'' functionals 
as follows. The proposition below is essentially a calculation of Bott-Chern secondary class 
(cf., \cite{BC65},\cite{Sou92}) and matches to the history of \cite{Don85},
\cite{Tia00b} and \cite
{Rub08}.

\begin{Prop}\label{global.local}

We follow the notation in the introduction and discuss arithmetic 
(relatively) projective variety 
$\mathcal{X}$. 
Suppose $\mathcal{L}_{i}  (i=0,1,\cdots,n)$ are (relatively) ample arithmetic 
line bundle on $\mathcal{X}$. 
If we change only infinity place part (i.e. the metric $h$), 
then we get a functional of the space of K\"ahler metrics as follows. 
Set 

$$\mathcal{G}\colon 
\Pi_{i=0}^{n}\mathcal{H}(\mathcal{L}_{i}(\mathbb{C}))\to \mathbb{R}$$ 
\begin{center}
as 
\end{center}
$$
\mathcal{G}(h_{0},\cdots,h_{n}):=
(\bar{\mathcal{L}}^{h_{0}}.\bar{\mathcal{L}}^{h_{1}}.\cdots.\bar{\mathcal{L}}^{h_{n}}), 
$$
simply the Arakelov-Gillet-Soul\'{e} theoretic intersection number. 
Then we have 

$$\mathcal{G}(e^{-2\varphi}\cdot h_{0},\cdots,h_{n})-\mathcal{G}(h_{0},\cdots,h_{n})
$$
$$
=\int_{X}\varphi\cdot c_{1}(\mathcal{L}_{1}(\mathbb{C}),h_{1})\wedge\cdots 
\wedge c_{1}(\mathcal{L}_{n}(\mathbb{C}),h_{n}). 
$$

\end{Prop}

\begin{proof}

Take general meromorphic sections $s_{i}$ of $\mathcal{L}_{i}(\mathbb{C})  (i=0,1,\cdots,n)$ 
which do not have common components. Then, the difference of two first arithmetic chern classes 
can be expressed 
$$c_{1}^{\hat{}}(\mathcal{L}_{0},e^{-2\varphi}\cdot h_{0})
-c_{1}^{\hat{}}(\mathcal{L}_{0},h_{0})=
\overline{(0,2\varphi)}\in \hat{\it CH}^{1}(\mathcal{X}), $$
by considering the same meromorphic section $s_{1}$. 
Hence, 
$$\mathcal{G}(e^{-2\varphi}\cdot h_{0},\cdots,h_{n})-\mathcal{G}(h_{0},\cdots,h_{n})$$
$$
=c_{1}^{\hat{}}(\mathcal{L}_{n},h_{n}).\cdots.c_{1}^{\hat{}}(\mathcal{L}_{1},h_{1}).(0,2\varphi)
$$
$$
=\int_{X}\varphi\cdot c_{1}(\mathcal{L}_{1}(\mathbb{C}),h_{1})\wedge\cdots 
\wedge c_{1}(\mathcal{L}_{n}(\mathbb{C}),h_{n}), 
$$
by the commutativity of the 
Gillet-Soul\'{e} intersection pairing (``(higher) Weil reciprocity''). 
\end{proof}


The resemblance of some algebro-geometric intersection theoretic invariants 
and functionals of the space of K\"ahler metrics such as the 
Aubin-Mabuchi (Monge-Amp\'{e}re) energy and the K-energy, about which we continue discussion below,  
are first discussed properly by Boucksom-Hisamoto-Jonsson \cite{BHJ15a}. 
A version of a part i.e. for Mong\'e-Amper\'e 
energy case is also 
in \cite{Oda15a} (still after the fruitful discussions with S.Boucksom in $2014$) in somewhat 
generalised setting compared with \cite{BHJ15a}.

\subsection{Modular height for general arithmetic schemes}\label{DF.sec}


Before the introduction of the Arakelov theoretic (global) version, 
we recall the classical K-energy (\cite{Mab86}) through the formula by Chen \cite{Che00}, Tian 
\cite{Tia00} which we regard here as the definition. We simultaneously recall the definitions of 
Ricci energy, Aubin-Mabuchi energy and entropy (cf., e.g., \cite{BBGZ13}) which form parts of them. 

\begin{Def}[{\cite{Mab86},\cite{Che00},\cite{Tia00},\cite{BBGZ13}}]
\label{funct}
Keeping the above notation, we recall the following notions: 
$$\mu_{\omega}(\varphi):=\frac{\bar{S}}{n+1}\mathcal{E}_{\omega}(\varphi)-
\mathcal{E}^{\it Ric(\omega)}(\varphi)+
\frac{1}{V}{\it Ent}_{\omega}(\omega_{\varphi}) \text{  (K-energy)},$$ 

\noindent
where $\bar{S}$ is the average scalar curvature, $V$ is the volume 
$\int \omega_{h}^{n}$ and 

$$\mathcal{E}_{\omega}(\varphi):=\frac{1}{V}\sum_{0\le i\le n}\int \varphi\omega^{i}\wedge\omega_{\varphi}^{n-i} \text{  (Aubin-Mabuchi energy)}, $$

$$\mathcal{E}^{\it Ric(\omega)}(\varphi):=\frac{1}{V}\sum_{0\le i\le n-1}\int \varphi {\it Ric}(\omega)\wedge\omega^{i}\wedge\omega_{\varphi}^{n-1-i} \text{  (Ricci energy)},$$ 

$${\it Ent}_{\omega}(\varphi):=\int {\it log}\biggl(\frac{\omega_{\varphi}^{n}}{\omega^{n}}\biggr)
\omega_{\varphi}^{n}
\text{ (entropy)}.$$ 
\end{Def}

\subsubsection{Definition}

We require the curvature of $(L,h)$ to be positive. 
Here is the definition of our main invariant \textit{(K-)modular height} $h_{K}$. 

\begin{Def}[Modular height]\label{aDF}
Suppose $(\mathcal{X},\mathcal{L},h)$ is an arithmetic projective scheme over 
${\it Spec}(\mathcal{O}_{K})$, satisfying the conditions in 
\textit{Notation and Conventions} (\ref{Not.Cov}). Then we define 

$$h_{\it K}(\mathcal{X},\mathcal{L},h):=$$
$$\frac{1}{[K:\mathbb{Q}]}
\bigl(-n(L^{n-1}.K_{X})
((\bar{\mathcal{L}}
^{h})^{n+1})+(n+1)(L^{n})((\bar{\mathcal{L}}^{h})^{n}.\overline{K_{\mathcal{X}/B}}^{\it Ric(\omega_{h})})\bigr).$$ 

Here we recall that $(X,L)$ is the generic fiber of our $(\mathcal{X},\mathcal{L})$ as we follow the 
notation (\ref{Not.Cov}). 
We remark (again) that subscript ``$K$'' of $h_{K}$ is put after the name of K-stability (thus, it is K of 
``\textit{K}\"ahler'') 
and \textit{not} our base field $K$. Indeed, it is easy to see that our definition 
does not depend on the base field $K$, i.e. finite extension of $K$ makes no change. 
We also recall here (as we defined in the ``notation'' part in the introduction) 
the metric on $K_{\mathcal{X}_{\eta}}$ is induced from $h$, 
i.e. the determinant metric of the metric induced by $\omega_{h}:=c_{1}(L_{\infty},h)$. 
Also, the author is happy to acknowledge here that R.Berman taught the author in 
May of 2016 that in 2012 he had found essentially the same definition as \ref{aDF}  
and obtained a closely result to \ref{h.lim.hk}, which we discuss later. 

\end{Def}

Here, and throughout this paper, $\bar{}$ means metrisations of the line bundles. 
The above definition can be extended to integrable adelic metrics on $L$ and $K_{X_{gen}}$ by 
the extended Arakelov intersection theory \cite{Zha95b}, 
which fits to the philosophy of \cite{Don10} etc. 

\begin{Prop}[archimedean rigidity]\label{A.rig}
For any $c\in \mathbb{R}$, we have 
$h_{\it K}(\mathcal{X},\mathcal{L},e^{2c}\cdot h)=
h_{\it K}(\mathcal{X},\mathcal{L},h).$ 
\end{Prop}

As this follows from fairly straightforward short calculations, we omit the proof. 
If we replace, $c$ by arbitrary positive real function, 
we can not see any reason that this does not change. The above is 
an analogue of the ``rigidity'' of \cite[4.6]{FR06} and indeed the proof follows 
exactly the same way. We also have the following non-archimedean version which is then completely 
similar to \cite[4.6]{FR06}. 

\begin{Prop}[Non-archimedean rigidity]\label{nA.rig}
For any Cartier divisor $D$ on $C$, we have 
$h_{\it K}(\mathcal{X},\mathcal{L}(\pi^{*}D),h)=
h_{\it K}(\mathcal{X},\mathcal{L},h).$ 
\end{Prop}

We avoid to write down the easy proof by the same reason as above. 

There is a subtle issue about the definition of Donaldson-Futaki invariant 
(which we inherit here) of families with non-reduced closed fibers. The author believes that 
idealistically one can define Donaldson-Futaki invariant after semi-stable reduction, 
but to avoid confusion and to follow the custom of this field (cf. \cite{Don02},\cite{Wan12},\cite{Oda13b}), 
we do not make a change. 

For KE case, by simply substituting terms of Definition \ref{aDF}, we get 

\begin{Prop}[Special case of modular height]\label{KE.hk}
Suppose $K_{X}\equiv aL$ with some $a\in \mathbb{R}$. 
Then 
$$h_{\it K}(\mathcal{X},\mathcal{L},h)
=\frac{(L^{n})}{[K:\mathbb{Q}]}\bigl((\overline{\mathcal{L}}^{h})^{n}.(n+1)\overline{(K_{\mathcal{X}/B})}^{\it Ric(\omega_{h})}-na\overline{\mathcal{L}}^{h})\bigr). 
$$

\end{Prop}

An important property of the above invariant is that 

\begin{Prop}\label{ADF.K}
$h_{\it K}(\mathcal{X},\mathcal{L},e^{-2\varphi}\cdot h)-
h_{\it K}(\mathcal{X},\mathcal{L},h)
=
\frac{(L^{n})}{[K:\mathbb{Q}]}\cdot \mu_{\omega_{h}}(\varphi)$, 
where $\mu$ denotes the Mabuchi K-energy \cite{Mab86}. 
\end{Prop}

\begin{proof}
We obtain the proof as a special case of Proposition \ref{global.local}, applied to the definition. 
\end{proof}

The above Proposition \ref{ADF.K} refines a Bott-Chern interpretation of K-energy \cite[p.215]{Tia00b} 
and shows that, with a precise meaning, \\ 
``\textit{Mabuchi's K-energy \cite{Mab86} is essentially an intersection number.}'' 
Special cases of a variant of the arithmetic Donaldson-Futaki invariant (modular height) 
are treated in the following literatures.

\begin{Ex}
As an important example, we explain that 
the Faltings' ``moduli-theoretic height'' \cite{Fal83} for abelian variety, 
whose metrical structure corresponds to the (Ricci-)flat K\"ahler metric, 
is the special case of our modular height $h_{K}(\mathcal{X},\mathcal{L},h_{KE})$. 
First we recall its original definition. 

\begin{Def}[\cite{Fal83}]\label{def.Falt.ht}
Suppose $(\mathcal{X}^{oo},\mathcal{L}^{oo})$ is a semi-abelian scheme of relative dimension $n$ 
over $\mathcal{O}_{K}$ which has a proper generic fiber 
$(X,L)$. We denote the 
zero section as $\epsilon\colon C\to \mathcal{X}$. We consider 
$\epsilon^{*}K_{\mathcal{X}^{\it oo}/C}$ and 
metrize it by $(\alpha,\bar{\alpha})_{\it Falt}:=(\frac{i}{2})^{n}\int_{X_{\infty}}
\alpha\wedge\bar{\alpha}$. 

Then we set $h_{\it Falt}(\mathcal{X},\mathcal{L}):=\frac{1}{[K:\mathbb{Q}]}
\hat{\it deg}(\epsilon^{*}K_{\mathcal{X}^{\it oo}/C})$, and call it 
{\it Faltings modular height}. \footnote{We need to be careful \textit{not} to be confused by the other 
``Faltings height'' introduced in \cite{Fal91}. cf., subsection \ref{MA}. }
\end{Def}

Then we observe that this is a special case of our modular height in the following sense. 

\begin{Thm}\label{Falt.ht}
In the above situation \ref{def.Falt.ht}, we set 
$(\mathcal{X},\mathcal{L})$ as 
a relative compactification
\footnote{i.e. open immersion to \textit{proper} scheme $\mathcal{X}$ over $C$. 
We also assume $\mathcal{L}$ is (relatively) ample over $C$. } 
of the N\'eron model $(\mathcal{X}^{o},\mathcal{L}^{o})$ of 
$(X,L)$ 
such that ${\it codim}((\mathcal{X}\setminus \mathcal{X}^{0})\subset \mathcal{X})\ge 2$. 
Then we have $$h_{K}(\mathcal{X},\mathcal{L},h_{KE})=(n+1)(L^{n})\cdot 
\biggl(h_{\it Falt}(X)+\frac{1}{2}{\it log}\biggl(\frac{(L^{n})}{n!}\biggr)\biggr),$$ 
where $h_{\it Falt}(-)$ denotes the Faltings ``modular-theoretic height" \cite{Fal83}. 
\end{Thm}

Later, as a special case of Theorem \ref{hk.min} via ``birational geometry'', we also find that 
this $(\mathcal{X},\mathcal{L},h_{KE})$ minimizes $h_{K}$ among all models, 
so that ``canonicity'' of Faltings height follows. 

\begin{proof}[proof of Theorem \ref{Falt.ht}]
We write the unique 
Ricci flat metric $g_{\it KE}$ whose normalised K\"ahler form $\omega_{\it KE}$ sits in 
$c_{1}(L_{\infty})$, 
and denote its determinantal metric on $K_{X_{\infty}}$ by ${\it det}(g_{KE})$. 
Then if we set $K_{\mathcal{X}/C}=\pi^{*}D$ with some arithmetic divisor $D$ on $C$, 
we observe that 
\begin{equation}\label{hk.falt1}
h_{K}(\mathcal{X},\mathcal{L},h_{\it KE})=(L^{n})\hat{{\it deg}}(\mathcal{O}_{C}(D,h))
\end{equation}
with hermitian metric 
$h$ which satisfies $h(\alpha,\beta)=(\alpha_{x},\beta_{x})_{{\it det}(g_{KE})}$ 
for any $x\in X_{\infty}=\mathcal{X}(\mathbb{C})$. 
Note also that this corresponds to 
the fundamental equality of the 
Deligne pairing (cf., e.g., \cite[p.79]{Zha95b}) 
$$<\bar{\mathcal{L}},\cdots,\bar{\mathcal{L}},
\overline{K_{\mathcal{X}/C}}>=\mathcal{O}_{C}((L^{n})D).$$

For $\alpha\in \Gamma(K_{X_{\infty}})$ and $x\in X_{\infty}=\mathcal{X}(\mathbb{C})$, we have 
$$(\alpha_{x},\alpha_{x})_{{\it det}(g_{KE})}\cdot \omega_{g_{\it KE}}^{n}
=(-1)^{\frac{n(n-1)}{2}}\cdot \biggl(\frac{i}{2}\biggr)^{n}\cdot n!(\alpha\wedge \bar{\alpha}).$$

This can be confirmed by an easy local calculation for the K\"ahler-Einstein metric, 
indeed the above holds for any K\"ahler metric. 
Then integrating the above over $X_{\infty}$, we have 
\begin{equation}\label{hk.falt2}
\frac{(L^{n})}{n!}\cdot (\alpha_{x},\alpha_{x})_{{\it det}(g_{KE})}
=(-1)^{\frac{n(n-1)}{2}}\cdot \biggl(\frac{i}{2}\biggr)^{n}\int_{X_{\infty}}\alpha\wedge \bar{\alpha}, 
\end{equation}
by the Ricci-flatness i.e. 
constancy of the quantity $(\alpha_{x},\alpha_{x})_{{\it det}(g_{KE})}$. 
Hence, combining together with (\ref{hk.falt1}), we get the assertion. 
Also note that essentially this derives (for geometric family, just by applying the above fiberwise)
the well-known potential description of the Weil-Petersson metric (e.g. \cite[Theorem2]{Tia87b}).

\end{proof}
\end{Ex}

\begin{Ex}[\cite{Fal84b}]
The case of arithmetic surface $\mathcal{X}\to C$ whose generic fiber $X$ has genus $g>1$, 
with the Arakelov-Faltings metric attached is the one essentially treated since \cite{Fal84b}. 
In particular, Faltings showed 
$$(\overline{\omega_{\mathcal{X}/C}}^{\it Ar}.\overline{\omega_{\mathcal{X}/C}}^{\it Ar})\ge 0$$ 
at \cite[Theorem 5(a)]{Fal84b} which is loosely related to \textit{Arakelov K-semistability} 
(of hyperbolic curve) which we introduce in the next section. Indeed, we have 
$$h_{K}(\mathcal{X},\mathcal{L}:=K_{\mathcal{X}/C},h^{\it Ar})=\frac{2g-2}{[K:\mathbb{Q}]}
(\overline{\omega_{\mathcal{X}/C}}^{\it Ar}.\overline{\omega_{\mathcal{X}/C}}^{\it Ar}),$$ 
\noindent
where $h^{\it Ar}$ is a metric on $K_{X_{\infty}}$ corresponding to the Arakelov metric. 
\end{Ex}

\begin{Ex}[Towards logarithmic setting]
With \cite{Don12} in mind, it is natural to think of generalisation to ``log'' setting i.e. to think of 
effective $\mathbb{R}$-Cartier divisor $\mathcal{D}$ in $\mathcal{X}$ and correspondingly 
canonical K\"ahler metrics with conical singularities along $\mathcal{D}(\mathbb{C})$. 
Once there would be an appropriate Arakelov intersection theory for such singular metrics, 
it would be straightforward to give the definitions of the generalising modular height (and other 
invariants/functionals which we will discuss later) as in the algebro-geometric situation \cite{OS15}. 
For example, Montplet \cite{Mont} essentially treats such 
``\textit{log-Arakelov-Donaldson-Futaki invariant}'' 
for pointed stable curves case by applying an extended Arakelov intersection theory \cite{BKK07}. 
For general definition of such log extension, we leave to future. 
\end{Ex}

Morally speaking, we show the decomposition of modular heights to places of number field
$$
\text{modular height } h_{K}=\int_{\it places} \text{(local) K-energy}. 
$$
We extend this picture to other kinds of intersection theoretic invariants 
from $\S 2.4$.


\subsection{Modular height and birational geometry (MMP)}

In the geometric setting i.e. when the base $C$ is a complex curve, 
the K-stability or more general behaviour of Donaldson-Futaki invariants are observed to 
be crucially controlled by the Minimal Model Program based birational geometry in \cite{Oda11a}, \cite{Oda13} 
and later developed in \cite{LX14},\cite{WX14},\cite{Der14},\cite{BHJ15a} etc. In this subsection, 
we partially establish an arithmetic version of the phenomenon. 

\subsubsection{Minimizing modular heights}

Our theorem below \ref{hk.min} (partially) justifies via our $h_{K}$ 
a speculation of Manin \cite[p76]{Man85} 
in 1980s shortly after \cite{Fal83} 
\begin{quotation}
\textit{``Our limited understanding of A-geometry (=Arakelov geometry) suggests the 
special role of those A-manifolds (``="Arakelov variety) 
for which $(X_{v},\omega_{v})$ are K\"ahler-Einstein. 
This condition appears to be a reasonable analog of the minimality of $X_f$ over 
${\it Spec}(R)$" 
\begin{flushright} -Y.~Manin, 1984\end{flushright}}
\end{quotation}

\noindent
Here, in his notation, $R$ is our $\mathcal{O}_{K}$, $X_f$ corresponds to our $\mathcal{X}$ as 
it is a 
finite type, proper, flat, surjective generically smooth scheme over $C={\it Spec}(R)$, 
$v$ is a place of $K$, $\omega_{v}$ is a K\"ahler form on the 
$v$-component of the complex variety $X_{\infty}$ (in our notation). 
Now we come back to our notation and state a justification of the above sentences. 

Roughly speaking, the below says that for a fixed generic fiber, 
if we take a sort of ``arithmetic minimal (or canonical) model'' from the MMP perspective as a scheme, 
and associate canonical K\"ahler metric such as K\"ahler-Einstein metric, the model minimizes 
the modular height $h_{K}$ among all possible models. In other words, such ``Arakelov \textit{minimal}  
model'' can be partially justfied by \textit{minimality} of the modular height $h_{K}$. 

\begin{Thm}\label{hk.min}

\begin{enumerate}
\item (Calabi-Yau case) \label{CY.min}
Let $(\mathcal{X},\mathcal{L})$ be a log terminal arithmetic polarized 
projective flat scheme of $(n+1)$-dimension, 
over $C:={\it Spec}(\mathcal{O}_{K})$ such that $K_{\mathcal{X}/C}$ is relatively numerically 
trivial and $(\mathcal{X},\mathcal{X}_{c}=\pi^{-1}(c))$ are 
log-canonical (resp., purely log-terminal) for any closed point $c\in C$. 
We metrize $\mathcal{X}(\mathbb{C})=:X_{\infty}$ with the unique 
(singular) K\"ahler-Einstein metric and take its corresponding 
continuous hermitian metric of $L_{\infty}=\mathcal{L}(\mathbb{C})$ 
which we denote as $h_{KE}$ (\cite{EGZ09},\cite{EGZ11}). 
For any other flat polarized family  $(\mathcal{X}',\mathcal{L}')\rightarrow C$ whose generic fiber is isomorphic to that of 
$\mathcal{X}\rightarrow C$ with possibly different metric $h$, we have 
$$
h_{K}(\mathcal{X},\bar{\mathcal{L}}^{h_{KE}})\leq (resp.\ <)
h_{K}(\mathcal{X}',\bar{\mathcal{L}'}^{h}). 
$$

\item \label{can.mod.min} (Canonical model case) 
Let $\mathcal{X}$ be a log-terminal arithmetic projective flat scheme of 
$(n+1)$-dimension over $C:={\it Spec}(\mathcal{O}_{K})$, the ring of integers of 
a number field $K$ such that $K_{\mathcal{X}/C}$ is relatively ample over $C$ and 
we set $\mathcal{L}:=\mathcal{O}_{\mathcal{X}}(mK_{\mathcal{X}/C})$ 
with sufficiently divisible $m\in \mathbb{Z}
_{>0}$\footnote{$m$ is unessential parameter due to the homogeneity of $h_{K}$ i.e. 
$h_{K}(\bar{\mathcal{L}}^{\otimes c})=c^{2n}h_{K}(\bar{\mathcal{L}})$}. Assume 
$(\mathcal{X},\mathcal{X}_{c})$ are log canonical pairs for any closed point $c\in C$. 
We metrize $\mathcal{X}(\mathbb{C})=:X_{\infty}$ with the 
(singular) K\"ahler-Einstein metric and take its corresponding 
continuous hermitian metric of $L_{\infty}=\mathcal{L}(\mathbb{C})$ 
which we denote as $h_{KE}$ (\cite{EGZ09},\cite{EGZ11}). 
Then, for any other flat polarized family  $(\mathcal{X}',\mathcal{L}')\rightarrow C$ whose generic fiber is  isomorphic to that of 
$\mathcal{X}\rightarrow C$ with possibly different metric $h$, we have 
$$
h_{K}(\mathcal{X},\bar{\mathcal{L}}^{h_{KE}})<
h_{K}(\mathcal{X}',\bar{\mathcal{L}'}^{h}). 
$$

\item (Special $\mathbb{Q}$-Fano varieties case) \label{sp.Fano.min}
Let $\mathcal{X}$ be a log-terminal arithmetic projective flat scheme of $(n+1)$-
dimensional over $C:={\it Spec}(\mathcal{O}_{K})$, the ring of integers of a number field $K$ such that 
$-K_{\mathcal{X}/C}$ is relatively ample and we set 
$\mathcal{L}=\mathcal{O}_{\mathcal{X}}(mK_{\mathcal{X}/C})$ 
with some sufficiently divisible $m\in \mathbb{Z}_{>0}$. Suppose that 
$\it glct((\mathcal{X},\mathcal{X}_{c});-K_{\mathcal{X}})$ which is defined as 
$${\it sup}\{t\ge 0 \mid (\mathcal{X},\mathcal{X}_{c}+tD) \text{ is lc for all effective } 
D\equiv_{/C} -K_{\mathcal{X}/C}\}
$$ 
are at least (resp.\ bigger than) $\frac{n}{n+1}$ for any closed point $c\in C$. 
From the theorem of Tian 
(\cite{Tia87a}),  
we know the existence of (singular) 
K\"ahler-Einstein metric on $\mathcal{X}(\mathbb{C})$ 
and we denote its corresponding hermitian metric of $\mathcal{O}_{X}(-mK_{X})$ by $h_{KE}$. 
Then for any other flat polarized family  
$(\mathcal{X}',\mathcal{L}')\rightarrow C$ whose geometric generic fiber is isomorphic to that of 
$\mathcal{X}\rightarrow C$ with possibly different hermitian metric $h$, we have 
$$
h_{K}(\mathcal{X},\bar{\mathcal{L}}^{h_{KE}})\leq (resp.\ <)
h_{K}(\mathcal{X}',\bar{\mathcal{L}'}^{h}). 
$$
\end{enumerate}
\end{Thm}

The theorem above \ref{hk.min} can be also regarded as a 
``de-quantized''\footnote{in the sense after Donaldson \cite{Don01}. 
In this paper, ``quantization" refers to this sense. } 
version of Zhang's (\cite{Zha96}) Chow-stable reduction, 
and is obtained as an arithmetic version of \cite[Theorem 6]{WX14},\cite[section 4]{Oda12}. 
The log terminality condition on general fiber/total space above is put 
to avoid the technical 
difficulty of K\"ahler-Einstein metrics on varieties with (semi-)log-canonicity such as \cite{BG14}'s. 
We take a somewhat complicated description via log pairs $(\mathcal{X},\mathcal{X}_{c})$ 
(cf., e.g., \cite{KM98}) 
as the inversion of adjunction in 
arithmetic setting is unfortunately not established yet. 
That is: 

\begin{Conj}[Inversion of adjunction]\label{inv.adj}
We keep the notation of \ref{Not.Cov} and set $\mathcal{X}_{c}=\pi^{-1}(c)$ for a 
closed point of $C$ as above. 
$D$ is an effective $\mathbb{Q}$-Cartier $\mathbb{Q}$-divisor whose 
support does not contain any component of $\mathcal{X}_{c}$, and $t$ is a non-negative 
real number. Then each of the following equivalences holds. 

\begin{enumerate}

\item $(\mathcal{X},\mathcal{X}_{c}+tD)$ is log-canonical in the neighborhood of $\mathcal{X}_{c}$ if and only if $(\mathcal{X}_{c},D|_{\mathcal{X}_{c}})$ is geometrically semi-log-canonical i.e.,  
semi-log-canonical after base change to algebraic closure of the residue field $\kappa(c)$ at $c$. 

\item $(\mathcal{X},\mathcal{X}_{c}+tD)$ is purely log terminal in the neighborhood of $\mathcal{X}_{c}$ 
if and only if $(\mathcal{X}_{c},D|_{\mathcal{X}_{c}})$ is log-terminal. 

\end{enumerate}

\end{Conj}

\noindent
Indeed, due to the recent progress 
such as \cite{CP14}, it looks hopeful to establish Conjecture \ref{inv.adj} up to $n=2$ in near  future. 

\begin{proof}[proof of Theorem \ref{hk.min}]
The basic strategy of the proofs is same as the equi-characteristic geometric case 
\cite{Oda11a},\cite{OS12},\cite[section 4]{Oda12} (for (\ref{can.mod.min}), also \cite{WX14}). 
For each case (1),(2),(3), we actually prove that the following holds: 
$$h_{K}(\mathcal{X},\bar{\mathcal{L}}^{h_{KE}})\leq (resp.\ <)
h_{K}(\mathcal{X},\bar{\mathcal{L}}^{h})\leq (resp.\ <)
h_{K}(\mathcal{X}',\bar{\mathcal{L}'}^{h}). $$

We prove these two inequalities separately and then the desired inequality will be 
obtained. 
The first inequality is a consequence of the nowadays standard fact that 
K-energy is minimized at K\"ahler-Einstein metrics when they exist (\cite{BM87},\cite{Che00b} etc). 
\footnote{One month after we posted this paper on arXiv, the author had a chance 
to learn a letter from S.~Zhang to P.~Deligne dated 3rd February 1993. In that letter, 
he essentially found K-energy for curve (arithmetic surface) case in the manner of our general formula \ref{h.lim.hk} 
later and showed that Poincar\'e metric minimizes it among the 
(archmedian) metrics. The author appreciates S-W.Zhang for teaching this.}

The proof of the second inequality is in the exactly same manner to that of 
\cite{Oda11a},\cite{OS12},\cite{Oda12}. 
For readers conveniences, we describe the proofs, partially referring to the original 
\cite{Oda11a},\cite{OS12},\cite{Oda12}. 
We first take a normal projective flat scheme $\mathcal{Y}$ which dominates 
both $\mathcal{X}$ and $\mathcal{X}'$ via birational morphisms. We denote the birational morphisms 
as $p\colon \mathcal{Y}\to \mathcal{X}$ and $q\colon \mathcal{Y}\to \mathcal{X}'$. 
As the model $\mathcal{Y}$ only needs to be normal (rather than regular), it can be easily obtained as the 
blow up of the indeterminancy ideal of the birational map $f\colon \mathcal{X}\dashrightarrow \mathcal{X}'$ 
(cf., e.g., \cite{Oda13}) or the normalisation of the graph of $f$. 
Then we prove the desired inequalities as consequences of comparing 
$\mathcal{L}$ on $\mathcal{X}$ and $q^{*}\mathcal{L}'$ on $\mathcal{Y}$, 
which is the main non-trivial part of the proof. 
In the original algebro-geometric settings of \cite[section 2]{Oda11a},\cite[section 3-5]{OS12}, 
our $(\mathcal{X},\mathcal{L})$ corresponds to the trivial test configuration so that its Donaldson-Futaki 
invariants were zero. But the completely same estimations techniques referred to as above, work 
to show that 
$\bigl(h_{K}(\mathcal{Y},q^{*}\bar{\mathcal{L}'}^{h})-h_{K}(\mathcal{X},\bar{\mathcal{L}}^{h})\bigr)>0$.

The details of each proofs are as follows. 
For the situation (\ref{CY.min}) of Calabi-Yau varieties, 
we have 
$$\bigl(h_{K}(\mathcal{Y},q^{*}\bar{\mathcal{L}'}^{h})-h_{K}(\mathcal{X},\bar{\mathcal{L}}^{h})\bigr)$$
$$=\frac{(n+1)(L^{n})}{[K:\mathbb{Q}]}
((q^{*}\bar{\mathcal{L}'}^{h})^{n}.\bar{K_{\mathcal{Y}/\mathcal{X}}})$$
and this is positive (resp., non-negative) if $K_{\mathcal{Y}/\mathcal{X}}$ is non-zero effective 
(resp., effective), since $q^{*}\bar{\mathcal{L}'}^{h}$ is vertically nef. 
This follows from the assumption of pure log terminality (resp., log-canonicity) 
of $(\mathcal{X},\mathcal{X}_{c})$. 
Heuristically, the above 
$((q^{*}\bar{\mathcal{L}'}^{h})^{n}.\bar{K_{\mathcal{Y}/\mathcal{X}}})$ 
is non-archimedean version of the entropy for Monge-Amp\'ere measures 
(cf., \cite{BHJ15a}, subsection \ref{Ar.Ent}) 
and originally we called ``discrepancy term'' in \cite{Oda11a},\cite{Oda13}. 

For the situation (\ref{can.mod.min}) of canonical models, 
following \cite[section 2]{Oda11a}, 
we decompose the quantity as 
$$\bigl(h_{K}(\mathcal{Y},q^{*}\bar{\mathcal{L}'}^{h})-h_{K}(\mathcal{X},\bar{\mathcal{L}}^{h})\bigr)$$
$$=\frac{(n+1)(L^{n})}{[K:\mathbb{Q}]}((q^{*}\bar{\mathcal{L}'}^{h})^{n}.
(p^{*}\bar{\mathcal{L}}^{h_{KE}})^{\otimes (n+1)}\otimes(q^{*}\bar{\mathcal{L}'}^{h})^{\otimes (-n)})$$
$$+\frac{(n+1)(L^{n})}{[K:\mathbb{Q}]}
((q^{*}\bar{\mathcal{L}'}^{h})^{n}.\bar{K_{\mathcal{Y}/\mathcal{X}}}).$$
As we saw above, the later term is non-negative because of the log-canonicity assumption 
for $(\mathcal{X},\mathcal{X}_{c})$ and we re-write the former term 
$((q^{*}\bar{\mathcal{L}'}^{h})^{n}.
(p^{*}\bar{\mathcal{L}}^{h_{KE}})^{\otimes (n+1)}\otimes(q^{*}\bar{\mathcal{L}'}^{h})^{\otimes (-n)})$ 
as follows as \cite[p2280, (2)]{Oda11a}. 
Heuristically, this part corresponds to 
the Aubin-Mabuchi energy plus the Ricci energy. 

For that, we prepare some more notation. 
We can and do assume $p^{*}\bar{\mathcal{L}}^{h_{KE}}=q^{*}\bar{\mathcal{L}'}^{h}(E)$ for 
some \textit{effective} divisor $E$, after replacing $\mathcal{L}$ by $\mathcal{L}(\pi^{*}D)$ for 
some Cartier divisor $D$ on $C$ ($\pi$ is the projection to $C$) if needed. Recall that 
replacement does not change $h_{K}$ by Proposition \ref{nA.rig} so are our terms which consist $h_{K}$. 
As $-E$ is $p$-ample from our construction, $p$ is a blow up along a closed subscheme $Z$. 
We set $s:={\it dim}(Z)$. Then we have 

$$((q^{*}\bar{\mathcal{L}'}^{h})^{n}.
(p^{*}\bar{\mathcal{L}}^{h_{KE}})^{\otimes (n+1)}\otimes(q^{*}\bar{\mathcal{L}'}^{h})^{\otimes (-n)})$$
$$=\biggl(-E^{2}.\sum_{i=1}^{n}(n+1-i+\epsilon_{i})((q^{*}\bar{\mathcal{L}'}^{h})^{n-i}.(p^{*}\bar{\mathcal{L}}^{h_{KE}})^{i-1}\biggr)$$
$$-\epsilon'((-E)^{n+1-s}.(p^{*}\bar{\mathcal{L}}^{h_{KE}})^{s}),$$

\noindent
for $0<|\epsilon_{i}|\ll 1 (1\le i\le n)$ and $0<\epsilon'\ll 1$ as in \cite[p2280, (2)]{Oda11a}. 
Thus this is positive by the same lemma \cite[Lemma 2.8]{Oda11a} as in geometric case. 
The lemma \cite[2.8]{Oda11a} can be proved in the same way once we replace the 
(equi-characteristic, geometric) Hodge index theorem by the Moriwaki Hodge index theorem (\cite{Mor96}). 

For the situation (\ref{sp.Fano.min}) of special $\mathbb{Q}$-Fano varieties, 
we decompose the quantity as 
$$\bigl(h_{K}(\mathcal{Y},q^{*}\bar{\mathcal{L}'}^{h})-h_{K}(\mathcal{X},\bar{\mathcal{L}}^{h})\bigr)$$
$$-((q^{*}\bar{\mathcal{L}'}^{h})^{n}.p^{*}\bar{\mathcal{L}}^{h_{KE}})+
((q^{*}\bar{\mathcal{L}'}^{h})^{n}.(n+1)mK_{\mathcal{Y}/\mathcal{X}}-
nE).$$

The first term $-((q^{*}\bar{\mathcal{L}'}^{h})^{n}.p^{*}\bar{\mathcal{L}}^{h_{KE}})$ 
is non-negative by \cite[4.3]{OS12}. The positivity (resp., non-negativity) of the second term 
$((q^{*}\bar{\mathcal{L}'}^{h})^{n}.(n+1)mK_{\mathcal{Y}/\mathcal{X}}-
nE)$ follows from the assumption on $\mathit{glct}$. 
Indeed the vertical divisor 
$(n+1)mK_{\mathcal{Y}/\mathcal{X}}-
nE$ is non-zero effective (resp., effective) under the $\mathit{glct}$ assumptions 
as the same argument of \cite[section 3]{OS12} (cf., also \cite[especially p7-9]{Oda11b}) 
so we omit its details. 

\end{proof}

Note that the above does \textit{not} use 
(arithmetic version of) the MMP conjecture \ref{unn.arith.MMP} \ref{n.arith.MMP} above 
but simple birational geometric arguments juggling with discrepancies. 
For (1),(2) of the above theorem \ref{hk.min}, 
it naturally extends to semi-log-canonical $\mathcal{X}$ 
but for technical difficulty to treat metrics of infinite diameters, we omit and do not claim it here. 
Also note that the left hand sides of the above inequalities concide with the 
``height on the quotient variety'' as in Zhang \cite{Zha96}, Maculan \cite{Mac14}. 


In general, for stable reduction, we naturally expect the following as 
the arithmetic counterpart of \cite[(4.3)]{Oda15a}. 

\begin{Conj}[Canonical reduction]
We fix a normal projective variety $(X,L)$ over a number field $K$ and 
consider all integral model i.e. $(\mathcal{X},\mathcal{L})$ over $\mathcal{O}_{K'}$ where 
$K'$ is a finite extension of $K$. 
If $(\mathcal{X},\mathcal{L})$ takes minimal $h_{K}$ among 
those while fixing $(X,L)$, i.e. 
$h_{K}(\mathcal{X},\mathcal{L})\le h_{K}(\mathcal{X}',\mathcal{L}')$ for any other 
integral projective model $(\mathcal{X}',\mathcal{L}')$ over a finite extension of $K$. 
Such ``$h_{K}$-minimising model'' satisfying the following properties: 
\begin{enumerate}
\item all the geometric fibers are reduced and semi-log-canonical. 
\item if the generic fiber $F$ is a klt $\mathbb{Q}$-Fano variety, then so are 
all fibers i.e. they are also klt $\mathbb{Q}$-Fano varieties. 
\item if $K_{F}\equiv a\mathcal{L}|_{F}$ with $a\ge 0$, 
$h_{K}(\mathcal{X},\mathcal{L})$ is minimum among all the models if and only if 
any fiber $G$ is  reduced and geometrically slc with $K_{G}\equiv a\mathcal{L}|_{G}$. 
\end{enumerate}
\end{Conj}

Comparing with the proof of geometric case, what lacks in the arithmetic situation is the 
presence of \textit{log} minimal model program as well as the stable reduction (in the sense of 
\cite[Chapter II]{KKMS73}) which could possibly take more than a decade. In the light of 
\cite{DS14},\cite{DS15} the above (especially (2)) can be seen as a sort of 
``arithmetic (pointed) Gromov-Hausdorff limits''. 


\subsubsection
{Decrease of Modular heights by semistable reduction and normalization}

Take an arbitrary principally polarized abelian variety 
$A$ over a number field $K$ and its base change $A^{(s)}$ to a finite extension $K'/K$ 
admitting semiabelian reduction, which exists due to Grothendieck and Deligne \cite{SGA7}. 
Then, it has been known that 

$$h_{\it Fal}(A)-h_{\it Fal}(A^{(s)})=\frac{1}{[K:\mathbb{Q}]}\sum_{\mathfrak{p}}c(A,\mathfrak{p})
{\it log}(N\mathfrak{p}),$$

\noindent
where $\mathfrak{p}$ runs over all prime ideals of 
$\mathcal{O}_{K'}$ and each 
$c(A,\mathfrak{p})$ is a positive real number 
called the \textit{base change conductor} by \cite{Chai00}. Indeed it follows from the definition of the Faltings height \cite{Fal83} that the above formula holds 
once we set 
$$c(A,\mathfrak{p}):=\frac{1}{e(\mathfrak{p};K'/K)}
{\it length}_{\mathcal{O}_{K'}}
\biggl(
\frac
{\Gamma({\it Spec}(\mathcal{O}_{K'}),\epsilon^{*}\omega_{\mathcal{A}
_{K'}/\mathcal{O}_{K'}}))}
{\Gamma({\it Spec}(\mathcal{O}_{K}),\epsilon^{*}
\omega_{\mathcal{A}/\mathcal{O}_{K}}))\otimes \mathcal{O}_{K'}}
\biggr),$$
where $e(\mathfrak{p};K'/K)$ is the ramifying index at $\mathfrak{p}$ 
and $\mathcal{A}$ 
(resp., $\mathcal{A}_{K'}$) 
means the N\'eron model of $A$ (resp., $A^{(s)}$). 
The fact that the base change conductor above is independent of 
the extension $K'$ follows easily from that after the 
semiabelian reduction at $\mathcal{O}_{K'}$, say if we 
extend it further as $K''/K'$, then 
$\mathcal{A}_{K'}\times_{K'}K''\hookrightarrow \mathcal{A}_{K''}$ 
is an open immersion. 
In particular, from the above observation we have 
$$h_{\it Fal}(A)\ge h_{\it Fal}(A^{(s)}).$$ 

We prove a somewhat analogous result for general arithmetic varieties 
as follows. 

\begin{Prop}\label{ss..red}
When $(\mathcal{X},\mathcal{L},h)\to {\it Spec}(\mathcal{O}_{K})$ 
is replaced by normalisation of the finite base change 
$\mathcal{O}_{K'}/\mathcal{O}_{K}$, i.e. to think of $(\mathcal{X}_{K'},\mathcal{L}_{K'},h_{K'}):=(\mathcal{X},
\mathcal{L},h)\times_{\mathcal{O}_{K}}\mathcal{O}_{K'}$ and denote it normalisation as 
$(\tilde{\mathcal{X}}_{K'},\tilde{\mathcal{L}}_{K'},\tilde{h}_{K'})$, then we have 
$$h_{K}(\tilde{\mathcal{X}}_{K'},\tilde{\mathcal{L}}_{K'},\tilde{h}_{K'})\le 
h_{K}(\mathcal{X},\mathcal{L},h).$$ 
\end{Prop}

\begin{proof}
This is not hard to prove and follows essentially from the simple fact that 
$K_{\tilde{\mathcal{X}}_{K'}/\mathcal{X}}:=
K_{\tilde{\mathcal{X}}_{K'}/\mathcal{O}_{K'}}-\nu^{*}K_{\mathcal{X}/\mathcal{O}_{K}}$, 
where $\nu\colon \tilde{\mathcal{X}}_{K'}\to \mathcal{X}$ denotes the finite morphism, is anti-effective. 
Indeed, as we assume the normality of the total space, thus of generic fiber, archimedean data does not 
affect. 
Thus, for 
$h_{K}(\mathcal{X},\mathcal{L},h)-
h_{K}(\tilde{\mathcal{X}}_{K'},\tilde{\mathcal{L}}_{K'},\tilde{h}_{K'})
=-\frac{(n+1)(L^{n})}{[K:\mathbb{Q}]}(\overline{\tilde{\mathcal{L}}_{K'}}^{n}.
\overline{K_{\tilde{\mathcal{X}}_{K'}/\mathcal{X}}})\ge 
0$ which completes the proof. 
\end{proof}

The above phenomoneon is essentially the one observed for geometric cases 
in \cite[5.1,5.2]{RT07},\cite[3.8]{Oda13},\cite[around 2.5, 4.3]{Oda15a} among others. 
It shows that, assuming the semistable reduction conjecture, we can always replace the 
integral model, after some extension of scalars, 
by those which has reduced fibers and less modular heights $h_{K}$. 

\subsubsection{Decrease of Modular heights by flow}

As a preparation, we consider the following natural generalisation of the MMP (with scaling) 
(cf., \cite{BCHM10}) 
to the arithmetic setting. 

\begin{Conj}[Arithmetic MMP with scaling]\label{unn.arith.MMP}
Starting from a log-canonical arithmetic projective variety $(\mathcal{X},\mathcal{L})$ with 
$K_{X}\equiv aL (a\in \mathbb{R})$ 
on the general fiber, we can run the semistable $K_{\mathcal{X}/C}$-MMP with 
scaling $\mathcal{L}$ as in \cite{BCHM10} (cf., also \cite{Fuj11}). 

More precisely, there is a sequence of birational modifications 
$\mathcal{X}=\mathcal{X}_{0}\dashrightarrow \mathcal{X}_{1}\dashrightarrow \cdots \mathcal{X}_{n}$ 
(each step is a flip or a divisorial contraction) so that the following holds 
\begin{enumerate}
\item 
There is a monotonely increasing sequence of real numbers $0=t_{0}<t_{1}<\cdots<t_{l}=\infty$ such that 
the strict transform of (a $\mathbb{R}$-divisor corresponding to) 
$\mathcal{L}(tK_{\mathcal{X}_{i}/C})$ $(t_{i-1}<t<t_{i})$ 
is (relatively) ample $\mathbb{R}$-line bundle 
over the base $C$. 
We naturally finish with $\mathcal{X}_{n}$ which is either a relative minimal model or a relative Mori 
fibration (cf., \cite{KM98}, \cite{BCHM10} for the basics). 
\item The above birational modifications are compatible with the 
K\"ahler-Ricci flow 
$$
\frac{\partial \omega_{t}}{\partial t}=-{\it Ric}(\omega_{t})    
$$
in the sense that $\omega_{t} (t_{i-1}\le t\le t_{i})$ 
are K\"ahler currents of $\mathcal{X}_{i}(\mathbb{C})$. 
\end{enumerate}

\end{Conj}

\noindent
Recent theory of the analytic MMP with scaling (cf., e.g., \cite{CL06} \cite{ST09}) 
show the compatibility written in (2), which was pioneered in \cite{Tsu88}, \cite{Tsu90}. 
For our purpose, we use the following normalised version. 

\begin{Conj}[Normalised arithmetic MMP with scaling]\label{n.arith.MMP}
Starting from a log-canonical arithmetic projective variety $(\mathcal{X},\mathcal{L})$ with 
$K_{X}\equiv aL (a\in \mathbb{R})$ 
on the general fiber, we can run the semistable $K_{\mathcal{X}/C}$-MMP with 
scaling $\mathcal{L}$ as in \cite{BCHM10} (cf., also \cite{Fuj11}). 

More precisely, there is a sequence of birational modifications 
$\mathcal{X}=\mathcal{X}_{0}\dashrightarrow \mathcal{X}_{1}\dashrightarrow \cdots \mathcal{X}_{n}$ 
(each step is a flip or a divisorial contraction) so that the following holds. 
There is a monotonely increasing sequence of real numbers $0=t_{0}<t_{1}<\cdots<t_{l}=\infty$ such that 
the strict transform of (a $\mathbb{R}$-divisor corresponding to) $\mathcal{L}(tE) (t_{i-1}<t<t_{i})$ 
is (relatively) ample over the base $C$. From this it naturally follows that $E$ is contracted (i.e. its 
strict transform vanishes) in $\mathcal{X}_{n}$ which is either a relative minimal model or a relative Mori 
fibration (cf., \cite{KM98}, \cite{BCHM10} for the basics). 

\end{Conj}

It is natural to simultaneously run the 
normalised K\"ahler-Ricci flow 
$$
\frac{\partial \omega_{t}}{\partial t}=-{\it Ric}(\omega_{t})+a\omega_{t}, 
$$
from some $\omega_{0}=c_{1}(\mathcal{L}(\mathbb{C}),h_{0})$ with some 
positively curved hermitian metric $h_{0}$ of $\mathcal{L}(\mathbb{C})$. 
Of course, the regular arithmetic surfaces case of the above conjecture \ref{n.arith.MMP} 
is settled 
by Lichtenbaum and is now classical \cite{Lich68}. 
Also the case of terminal threefolds with geometrically semi-log-canonical fibers
\footnote{Accurately, when $(\mathcal{X},\mathcal{X}_{c})$ is dlt for any closed point $c\in C$} 
is settled by Kawamata \cite{Kaw94},\cite{Kaw99}. Recently, H.~Tanaka \cite{Tan} have confirmed 
the extension of \cite{Lich68} to klt arithmetic surface case, in the modern MMP framework, as well. 

The following ``monotonely decrease" theorem has its origin in geometric counterpart as 
\cite{LX14}, \cite[4.1]{Oda15a}. 

\begin{Thm}\label{hK.decrease}
Assume the above conjecture \ref{n.arith.MMP}. 
For generically KE case, $h_{\it K}$ decreases along the arithmetic MMP with scaling 
in the sense of the above conjeture \ref{n.arith.MMP}. More precisely, if $(\mathcal{X},\mathcal{L})$ 
satisfies the condition of conjecture \ref{n.arith.MMP} and we define 
for $t_{i-1}\le t<t_{i}$ $F(t):=h_{K}(\mathcal{X}_{i},\mathcal{L}_{i},h_{i})$, it monotonely decrease 
i.e. for $t<s$, $F(t)>F(s)$. 
\end{Thm}

\begin{proof}
First we need to check that $(\mathcal{X}_{i},\mathcal{L}_{i},h_{i})$ satiefies our 
condition in Notation and Convention \ref{Not.Cov}. Only nontrivial parts for that are 
preservation of the normal $\mathbb{Q}$-Gorenstein property and the almost smoothness (in the sense of 
\ref{Not.Cov}) of the hermitian metrics. For the former half, the same proof of geometric case 
(cf., e.g. \cite{KM98}) applies. The latter is proven in \cite{ST09}. 
In the direction of $\bar{E}:=\overline{K_{\mathcal{X}/B}}^{\it Ric(\omega_{h})}-a\bar{\mathcal{L}}^{h}$, 
the derivation of $h_{\it K}$ is 
$\frac{dF(t)}{dt}=((\overline{\mathcal{L}}^{h})^{n-1}.\bar{E}^{2})$. 

The proof of decrease of \textit{finite i.e. non-archimedean part} of the 
modular heights $h_{K}$ along time development 
is completely the same as geometric case (\cite{LX14},\cite{Oda12},\cite{Oda15a}) and 
we only need to replace the use of the usual Hodge index theorem by 
the arithmetic Hodge index theorem \cite[Theorem B]{Mor96} (after 
Faltings-Hriljac). If we deal with integrable adelic metrics, 
we replace by the generalised arithmetic Hodge index theorem 
to those metrics due to \cite[Theorem1.3]{YZ13}. 

The infinite place part 
is nothing but the known fact that the K-energy decreases along the (normalised) K\"ahler-Ricci flow 
(cf., e.g. \cite{CT02}), 
which can be proved as follows: if we set $\varphi_{t}$ as ${\it Ric}(\omega_{t})+a\omega_{t}=i\partial\bar{\partial}\varphi$, 
then we have $$\frac{d\mu_{\omega_{0}}(\omega_{t})}{d t}=\int_{X_{\infty}}\biggl(\frac{d\varphi_{t}}{dt}|_{t=0}\biggr)
\frac{i\partial\bar{\partial}}{2}\biggl(\frac{d\varphi_{t}}{dt}|_{t=0}\biggr)\omega_{0}^{n-1}$$
$$=-\int_{X_{\infty}}\frac{i}{2}\biggl(\partial\frac{d\varphi_{t}}{dt}|_{t=0}\biggr)\overline{\biggl(\partial\frac{d\varphi_{t}}{dt}|_{t=0}\biggr)}\omega_{0}^{n-1}\leq 0,$$
as $-\frac{i}{2}\partial\psi\overline{\partial\psi}$ is semipositive $(1,1)$ form for 
arbitrary real function $\psi$. 
\end{proof}



\subsection{Arakelov energy}\label{MA}

We keep the same notation. The invariant we introduce in this subsection is 
essentially just a self-intersection number of Gillet-Soul\'e and had repeatedly 
appeared in various contexts before (\cite{Fal91}, \cite{Sou92}, \cite{RLV00}, \cite{BB10} etc).

\begin{Def}
The \textit{Arakelov (Aubin-Mabuchi) energy} is (simply!) defined as 
$\mathcal{E}^{Ar}(\mathcal{X},\bar{\mathcal{L}}(=(\mathcal{L},h))):=\frac{1}{[K:\mathbb{Q}]}(\bar{\mathcal{L}}^{h})^{n+1}$, an Arakelov-Gillet-Soul\'{e} intersection theory 
(cf., \cite{Sou92}). 
\end{Def}

As the name let us expect, the following holds. 

\begin{Prop}
$\mathcal{E}^{Ar}(\mathcal{X},\mathcal{L},e^{-2\varphi}\cdot h)
-\mathcal{E}^{Ar}(\mathcal{X},\mathcal{L},h)=
\frac{(L^{n})}{[K:\mathbb{Q}]}\cdot \mathcal{E}_{\omega_{h}}(\varphi), $
the Mong\'{e}-Amp\`{e}re energy. 
\end{Prop}

\begin{proof}
This follows as a special case of Proposition \ref{global.local} (i). 
\end{proof}

\noindent
In this case, the corresponding limiting analysis of \cite{BHJ15b} is also recently 
discussed in \cite{Rub15} especially for low dimensional case. 

Experts of Arakelov geometry should notice right away that, if $\mathcal{L}=\mathcal{O}(1)$ for an 
embedding $\mathcal{X}\subset \mathbb{P}(\mathcal{E})$ with arithmetic metrized bundle $\mathcal{E}$, 
this is essentially 
the \textit{(cycle's) height} treated in \cite{Fal91}
\footnote{hence this is also sometimes called ``Faltings height'' but we distinguish this with 
our $h_{K}$ crucially}, \cite{Sou92} etc.

Hence, in our language, 
the Cornalba-Harris-Bost-Zhang inequality gives a lower bound of 
this ``Arakelov Monge-Ampere energy'' of Chow (semi)stable varieties by some ``quantised'' invariant. 
Here, we allow some twist for the projective bundle in concern.

\subsection{Arakelov Ricci energy}\label{Ar.Ric.en}

We use a reference model 
$\pi_{\it ref}\colon (\mathcal{X}_{\it ref},\mathcal{L}_{\it ref}, h_{\it ref})\to 
{\it Spec}(\mathcal{O}_{K})$. 

\begin{Def}
We define the \textit{Arakelov Ricci energy} as follows. 
For $(\mathcal{X},\mathcal{L},h)$, we construct a common generic resolution 
$\tilde{\mathcal{X}}$ which is normal and 
dominates both models i.e. there are birational proper morphisms 
$p\colon \tilde{\mathcal{X}}\to \mathcal{X}$ and 
$q\colon \tilde{\mathcal{X}}\to \mathcal{X}_{\it ref}$. Then we set 
$$\mathcal{E}^{\it Ar.Ric}_{(\mathcal{X}_{\it ref},\mathcal{L}_{\it ref},h_{\it ref})}
(\mathcal{X},\mathcal{L},h):=\frac{1}{[K:\mathbb{Q}]}
\biggl(((p^{*}\bar{\mathcal{L}}_{\it ref}^{h_{\it ref}})^{n}-q^{*}\bar{\mathcal{L}}^{h})^{n}).\overline{K_{\tilde{\mathcal{X}}/C}}^{\it Ric(\omega_{h})}\biggr)$$ 
as an Arakelov intersection number on $\tilde{\mathcal{X}}$. It is easy to see that this does 
not depend on the choice of $\tilde{\mathcal{X}}$. 
\end{Def}

\begin{Prop}We have 
$\mathcal{E}^{Ar.Ric}_{(\mathcal{X}_{\it ref},\mathcal{L}_{\it ref},h_{\it ref})}
(\mathcal{X}_{\it ref},\mathcal{L}_{\it ref},e^{-2\varphi}\cdot h_{\it ref})
=\frac{(L^{n})}{[K:\mathbb{Q}]}\mathcal{E}^{\it Ric}_{\omega_{h_{\it ref}}}(\varphi), $
where $\mathcal{E}^{\it Ric}$ denotes the Ricci energy (see Definition \ref{funct}). 
\end{Prop}

\begin{proof}
Again, this follows as a special case of Proposition \ref{global.local} (i). 
\end{proof}

\subsection{Entropy}\label{Ar.Ent}

Arakelov entropy is defined as follows. 
Again, we use the (same) reference model 
$\pi_{\it ref}\colon (\mathcal{X}_{\it ref},\mathcal{L}_{\it ref}, h_{\it ref})\to 
{\it Spec}(\mathcal{O}_{K})$.

\begin{Def}
For $(\mathcal{X},\mathcal{L},h)$, we construct a model $\tilde{\mathcal{X}}$ 
which dominates both models i.e. there are birational proper morphisms $p\colon \tilde{\mathcal{X}}\to \mathcal{X}$ and $q\colon \tilde{\mathcal{X}}\to \mathcal{X}_{\it ref}$. Then we set 
${\it Ent}^{Ar}_{(\mathcal{X}_{\it ref},\mathcal{L}_{\it ref},h_{\it ref})}
(\mathcal{X},\mathcal{L},h):=$
$$\frac{1}{[K:\mathbb{Q}]}((p^{*}\bar{\mathcal{L}}^{h})^{n}.
p^{*}\overline{K_{\mathcal{X}/C}}^{\it Ric(\omega_{h})}-q^{*}\overline{K_{\mathcal{X}_{\it ref}/C}}^
{\it Ric(\omega_{h_{\it ref}})}),$$
as the Gillet-Soul\'e intersection number (cf.,\cite{Sou92}). It is easy to see that 
this does not depend on the common resolution $\tilde{\mathcal{X}}$. 
\end{Def}

\begin{Prop}We have 
${\it Ent}^{Ar}_{(\mathcal{X}_{\it ref},\mathcal{L}_{\it ref},h_{\it ref})}
(\mathcal{X}_{\it ref},\mathcal{L}_{\it ref},e^{-2\varphi}\cdot h_{\it ref})
=\frac{(L^{n})}{[K:\mathbb{Q}]}{\it Ent}_{\omega_{h}}((\omega_{h}+dd^{c}\varphi)^{n}), $
where ${\it Ent}$ means the (usual) entropy (see Definition \ref{funct}). 
\end{Prop}

\begin{proof}
Again, this follows as a special case of Proposition \ref{global.local} (i). 
\end{proof}

The sum of the above three are the Arakelov Donaldson-Futaki invariant as follows as 
an analogue of Definition \ref{funct} 
and the Donaldson-Futaki invariants' formula \cite{Wan12}, \cite{Oda13b}. 

\begin{Prop}[Decomposing the modular height]
$$h_{\it K}(\mathcal{X},\mathcal{L},h)-h_{\it K}(\mathcal{X}_{\it ref},\mathcal{L}_{\it ref},h_{\it ref})$$
$$=
\frac{\bar{S}}{n+1}\mathcal{E}^{\it Ar}_{\omega}(\mathcal{X},\mathcal{L},h)-
\mathcal{E}_{\it (\mathcal{X}_{\it ref},\mathcal{L}_{\it ref},h_{\it ref})}^{\it Ar.Ric(\omega)}(\mathcal{X},\mathcal{L},h)+
\frac{{\it Ent}_{\it (\mathcal{X}_{\it ref},\mathcal{L}_{\it ref},h_{\it ref})}^{\it Ar}
(\mathcal{X},\mathcal{L},h)}{[K:\mathbb{Q}]},$$ where $\bar{S}$ is the average scalar curvature of $\omega_{h}$ of 
geometric generic fiber and $V$ is the volume of $\omega_{h}$. 
\end{Prop}

\noindent
The proof is straightforward from the definitions. 

\subsection{Arakelov Aubin functionals}\label{Ar.Aubin}

We recall the original Aubin functionals.
\begin{Def}[\cite{Aub84}]\label{Aub..}
\begin{enumerate}
For $\omega_{h}$-psh (pluri-sub-harmonic) smooth funtion $\varphi$, we set 

\item $I_{\omega_{h}}(\varphi):=\frac{1}{V}\int_{X_{\infty}}\varphi(\omega_{h}^{n}-\omega_{\varphi}^{n}),$ 

\item $J_{\omega_{h}}(\varphi):=\frac{1}{V}\int_{X_{\infty}}\varphi\omega_{h}^{n}-
\frac{1}{(n+1)V}\sum_{j=0}^{n}\int_{X}\varphi(\omega_{\varphi}^{j}\wedge\omega_{h}^{n-j}). $

\end{enumerate}
\end{Def}

Now we define the arithmetic (Arakelov) version of the Aubin functionals 
$\mathcal{I}^{Ar}, \mathcal{J}^{Ar}$as follows. Again, we use the (same) reference model 
$\pi\colon (\mathcal{X}_{\it ref},\mathcal{L}_{\it ref}, h_{\it ref})\to 
{\it Spec}(\mathcal{O}_{K})$ and keep the notation of previous subsections 
\ref{Ar.Ric.en} and \ref{Ar.Ent}. 

\begin{Def}
$$\mathcal{I}^{Ar}_{(\mathcal{X}_{\it ref},\mathcal{L}_{\it ref},h_{\it ref})}
(\mathcal{X},\mathcal{L},h):=\dfrac{1}{[K:\mathbb{Q}]}\times$$
$$\biggl(-(p^{*}\bar{\mathcal{L}}^{h})^{n+1}
-(q^{*}\bar{\mathcal{L}}_{\it ref}^{h_{\it ref}})^{n+1}+(p^{*}\bar{\mathcal{L}}^{\it h}.
(q^{*}\bar{\mathcal{L}}_{\it ref}^{h_{\it ref}})^{n})
+(q^{*}\bar{\mathcal{L}}_{\it ref}^{h_{\it ref}}.
(p^{*}\bar{\mathcal{L}}^{h})^{n})\biggr)$$

$$\mathcal{J}^{Ar}_{(\mathcal{X}_{\it ref},\mathcal{L}_{\it ref},h_{\it ref})}
(\mathcal{X},\mathcal{L},h):=$$
$$
\dfrac{1}{[K:\mathbb{Q}]}
\biggl((p^{*}\overline{\mathcal{L}}^{h}.(q^{*}\bar{\mathcal{L}}_{\it ref}^{h_{\it ref}})^{n})-
\frac{1}{n+1}
(q^{*}(\bar{\mathcal{L}}_{\it ref}^{h_{\it ref}})^{n+1})\bigr)+
\frac{n}{n+1}(p^{*}(\bar{\mathcal{L}}^{h})^{n+1})\biggr). 
$$
\end{Def}

\begin{Prop}For $\omega_{h_{\it ref}}$-psh smooth function $\varphi$, we have 
$\mathcal{I}^{Ar}_{(\mathcal{X}_{\it ref},\mathcal{L}_{\it ref},h_{\it ref})}
(\mathcal{X}_{\it ref},\mathcal{L}_{\it ref},e^{-2\varphi}\cdot h_{\it ref})
=\frac{(L^{n})}{[K:\mathbb{Q}]}\mathcal{I}_{\omega_{h_{\it ref}}}(\varphi), $
the (usual) Aubin functional. 
\end{Prop}

\begin{proof}
Again, this follows as a special case of Proposition \ref{global.local} (i). 
\end{proof}

One can  hope to use the above Arakelov-Aubin functional as a certain ``norm'' when estimating 
our modular height $h_{K}$, as in original algebro-geometric setting. 
For example, the recent theory of uniform K-stability \cite{Der14},\cite{BHJ15a} 
makes use of the Aubin functional \ref{Aub..}. 
Similarly to the original K\"ahler situation \cite[p146-147]{Aub84}, 
we have the following fundamental inequality for our arithmetic situation. 

\begin{Prop}
Keeping the notations, we have 
$$0\le \frac{1}{n+1}\mathcal{I}^{Ar}_{(\mathcal{X}_{\it ref},\mathcal{L}_{\it ref},h_{\it ref})}
\le \mathcal{J}^{Ar}_{(\mathcal{X}_{\it ref},\mathcal{L}_{\it ref},h_{\it ref})}
\le \frac{n}{n+1}\mathcal{I}^{Ar}_{(\mathcal{X}_{\it ref},\mathcal{L}_{\it ref},h_{\it ref})}.$$
\end{Prop}

\begin{proof}
Although the essential techniques are completely same as known classical case, 
we hope the following gives simpler explanation for readers. Indeed, for example, 
the corresponding estimates for test configurations are done in \cite[Proposition 7,8]{BHJ15a}. 

In our proof, we only use the Hodge index theorem in Arakelov-geometric setting due to 
\cite{Mor96}, \cite[Theorem 1.3]{YZ13}. We take value at $(\mathcal{X},\mathcal{L},h)$ of the functionals. 
We set $\mathcal{O}_{\tilde{\mathcal{X}}}(\bar{E}):=(p^{*}\bar{\mathcal{L}}^{h})\otimes(q^{*}\bar{\mathcal{L}}_{\it ref}^{h_{\it ref}})^{\otimes (-1)}$, with an Arakelov divisor $\bar{E}$. Then the desired inequalities can be rewritten after some simple calculations as 
\begin{enumerate}
\item \label{aub1} $\bigl(-\bar{E}^{2}.(p^{*}\bar{\mathcal{L}}^{h})^{n-1}
+(p^{*}\bar{\mathcal{L}}^{h})^{n-2}.(q^{*}\bar{\mathcal{L}}_{\it ref}^{h_{\it ref}})^{1}
+\cdots
+(q^{*}\bar{\mathcal{L}}_{\it ref}^{h_{\it ref}})^{n-1}\bigr)\ge 0. $
\item \label{aub2} $\bigl(-\bar{E}^{2}.\sum_{k=0}^{n-1}(n-k)(p^{*}\bar{\mathcal{L}}^{h})^{n-1-k}.
(q^{*}\bar{\mathcal{L}}_{\it ref}^{h_{\it ref}})^{k}\bigr)\ge 0. $
\item \label{aub3} $\bigl(-\bar{E}^{2}.\sum_{k=0}^{n-1}(n-k)
(q^{*}\bar{\mathcal{L}}_{\it ref}^{h_{\it ref}})^{n-1-k}.(p^{*}\bar{\mathcal{L}}^{h})^{k}\bigr)\ge 0. $
\end{enumerate}

Indeed the first inequality (\ref{aub1}) gives $\mathcal{I}^{Ar}\ge 0$, 
the second inequality gives 
$\frac{1}{n+1}\mathcal{I}^{Ar}
\le \mathcal{J}^{Ar},$ and the last inequality gives 
$\mathcal{J}^{Ar}\le \frac{n}{n+1}\mathcal{I}^{Ar}$. 

\end{proof}

We end this subsection by a remark that the arguments of 
\cite[(2.6),(2.7),(2.8)]{Oda11a} are via similar techniques and 
indeed give us similar following inequalities under the same notation as above. 
\begin{equation}
(n+1)((p^{*}\bar{\mathcal{L}}^{h})^{n}.
q^{*}\bar{\mathcal{L}}_{\it ref}^{h_{\it ref}})\ge n(p^{*}\bar{\mathcal{L}}^{h})^{n+1}
+(q^{*}\bar{\mathcal{L}}_{\it ref}^{h_{\it ref}})^{n+1}, 
\end{equation}
\begin{equation}
(n+1)((q^{*}\bar{\mathcal{L}}_{\it ref}^{h_{\it ref}})^{n}.
p^{*}\bar{\mathcal{L}}^{h})\ge n(q^{*}\bar{\mathcal{L}}_{\it ref}^{h_{\it ref}})^{n+1}
+(p^{*}\bar{\mathcal{L}}^{h})^{n+1}. 
\end{equation}

As we partially show later in section \ref{a.K.mod}, all the above functionals defined so far can be encoded as metrized line bundles over higher dimensional 
arithmetic base and partially on 
arithmetic moduli spaces (\textit{Arakelov K-moduli}). It will be done simply and straightforwardly by 
replacing the Gillet-Soul\'e intersection number we use here in our definitions by the Deligne pairings. 

\subsection{Non-archimedean scalar curvature and Calabi energy} 

To discuss the general case of constant scalar curvature K\"ahler metric (cscK) or 
more broadly extremal metrics in the sense of E.~Calabi, 
we certainly need discussions of scalar curvature and related functional such as the Calabi-energy. 
Regarding this kind of energies, even the non-archimedean analogues are not introduced yet 
as far as the author knows, hence 
we would like to start with it. In this section, our base is a 
\textit{smooth projective curve $C$ over 
a field} $k$ and $\pi\colon (\mathcal{X},\mathcal{L})\to C$ is a 
projective flat family of relative dimension $n$. For simplicity, 
we suppose $\mathcal{X}$ is normal and \underline{$\mathbb{Q}$-factorial} in this 
subsection. Temporarily, we do not discuss extension to adelic metrics 
and from here to the end of our paper, we do not mean that such extension is automatic anymore.

\subsubsection{Equi-characteristic situation}

For $\pi\colon(\mathcal{X},\mathcal{L})\to C$ with $\mathcal{X}_{0}=\cup_{i}E_{i}$, 
the \textit{non-archimedean scalar curvature} $S^{nA}$ is a function from the set of 
irreducible components of $\mathcal{X}_{0}$ defined as: 

$$S^{nA}\colon E_{i}\mapsto \frac{-n(\mathcal{L}|_{E_{i}}^{n-1}.K_{\mathcal{X}/C}|_{E_{i}})}{(\mathcal{L}|_{E_{i}}^{n})}.$$

We make a brief review of minimization of the (normalized) Donaldson-Futaki invariant from 
\cite[section 4]{Oda15a}. Fixing a general fiber $(X,L)$ over $\eta\in C$, we consider 
all polarized models $\pi\colon (\mathcal{X},\mathcal{L})\to \tilde{C}\to C$ where 
$\mathcal{X}\to \tilde{C}$ is projective and $\tilde{C}\to C$ is a finite covering. 
We called 
$$
\frac{{\it DF}(\mathcal{X},\mathcal{L})}{{\it deg}(\tilde{C}\to C)}
:=\frac{-n(L^{n-1}.K_{X})(\mathcal{L}^{n+1})+(n+1)(L^{n})(\mathcal{L}^{n}.K_{\mathcal{X}/\tilde{C}})}
{{\it deg}(\tilde{C}\to C)}
$$
the \textit{normalized Donaldson-Futaki invariant} of $(\mathcal{X},\mathcal{L}) (\to \tilde{C}\to C)$ 
and denoted by ${\it nDF}(\mathcal{X},\mathcal{L})$. 

Among all models $(\mathcal{X},\mathcal{L})$ over finite coverings $\tilde{C}$ of $C$, 
if $\pi$ is minimising the above normalised Donaldson-Futaki invariant, 
we \cite[(4.3)]{Oda15a} proved that $\mathcal{X}_{0}$ is 
reduced and only admits semi-log-canonical singularities. 
In some situations, we proved more 
(cf., \cite[section 4]{Oda15a} for more details). 
We add one more property of such family in our context: 

\begin{Prop}\label{ndf.min}
If $(\mathcal{X},\mathcal{L})\to \tilde{C}$ takes the 
minimal normalised-Donaldson-Futaki invariant among the 
models of $(X,L)$ over finite coverings $\tilde{C}$ of $C$, 
then the non-archimedean scalar curvature $S^{nA}$ is constant 
i.e. $S^{nA}(E_{i})$ does not depend on $i$. 
\end{Prop}

\begin{proof}
Supposing the contrary, then either $(\mathcal{X},\mathcal{L}(\epsilon E_{i}))$ with $0<\epsilon\ll 1$ or 
$(\mathcal{X},\mathcal{L}(-\epsilon E_{i}))$ 
with $0<\epsilon\ll 1$ has less (normalised) Donaldson-Futaki invariant than that 
of $(\mathcal{X},\mathcal{L})$.
Thus it contradicts to the fact that $(\mathcal{X},\mathcal{L})$ minimizes the (normalizing) 
Donaldson-Futaki invariant. 
\end{proof}
\noindent
The above proposition \ref{ndf.min} also extends an observation made for the Calabi-Yau case 
(\cite[4.2 (i)]{Oda12}). 

We also give a new way of interpretation of K-stability via scalar curvature of the 
$(n+1)$-dimensional total space of test configurations: 

\begin{DefProp}
A polarized projective variety $(X,L)$ is K-polystable if and only if the following $(*)$ holds. 

$(*)$: For any test configuration with $\mathbb{Q}$-line bundle of exponent one, 
if we think of a natural compactification 
\footnote{see \cite{Oda13b} etc for precise details} $(\mathcal{X},\mathcal{L})$ 
over $\mathbb{P}^{1}$ by attaching 
$(X,L)$ at $\infty\in \mathbb{P}^{1}$, supposing $\mathcal{L}$ is (absolutely) ample 
(we replace $\mathcal{L}$ by $\mathcal{L}(mF)$ with $m\gg 0$ to make it ample if not), 

$$\frac{-(n+1)(\mathcal{L}^{n}.K_{\mathcal{X}/C})}{(\mathcal{L}^{n+1})}\le 
\frac{-n(L^{n-1}.K_{X})}{(L^{n})}, $$

\noindent
with equality holds exactly when the $(\mathcal{X},\mathcal{L})$ is a (naturally compactified) 
product test configuration i.e. 
$(X,L)$-fiber bundle. Note that the above inequality remains equivalent even when change $m$. 

\end{DefProp}

\noindent 
The above new way of paraphrasing K-stability 
is analogus to slope theories (Mumford\cite{Mum62},Takemoto\cite{Tak72}, Ross-Thomas \cite{RT07}) and 
follows straightforward from the general formula for the 
Donaldson-Futaki invariant \cite{Wan12}, \cite{Oda13}. 
Note that the left hand side is a sort of 
scalar curvature average of $(\mathcal{X},\mathcal{L})$ and the right hand side is 
precisely the scalar curvature average of $(X,L)$. 

The above re-interpretation of K-stability ``via average scalar curvatures'' 
was found during discussions with R.Thomas in 2013. 

\subsubsection{Non-archimedean Calabi functional}

Inspired by the observation in the above arguments, let us propose 
our working definition of non-archimedean Calabi functional and 
Arakelov Calabi functional as follows. We will come back to further study of these 
in future. 

\begin{Def}[Non-archimedean Calabi functional]
For a projective flat family $\mathcal{X}$ over a smooth proper curve $C$ with a relatively ample 
line bundle $\mathcal{L}$ and finite closed points set $S\subset C(k)$, we set 
$$\mathit{Ca}_{S}(\mathcal{X},\mathcal{L}):=\sum_{i}
\biggl(\frac{(\mathcal{L}|_{E_{i}}^{n-1}.K_{\mathcal{X}/C})}{(\mathcal{L}|_{E_{i}})^{n}}\biggr)^{2}(>0),$$
where $\cup_{i}E_{i}={\it Supp}(\cup_{s\in S}\mathcal{X}_{s})$ is the 
irreducible decomposition of the support of union of the special fibers 
$\mathcal{X}_{s}$. 
\end{Def}

\begin{Def}[Arakelov-Calabi functional]
Suppose $\mathcal{X}$ is a regular variety which is projective over $C:={\it Spec}(\mathcal{O}_{K})$, 
the ring of integers of $K$, with a relatively ample 
line bundle $\mathcal{L}$ with a hermitian metric $h$. 
For a finite set $S$ of places of $K$, we decompose it as $S=S^{\it fin}\cup S^{\infty}$ 
to the finite places and infinite places. 
Then we set $\mathit{Ca}^{\it Ar}(\mathcal{X},\bar{\mathcal{L}}:=(\mathcal{L},h))$ as 
$$\sum_{\cup_{i}E_{i}=\cup_{s\in S^{\it fin}}\mathcal{X}_{s}}
\biggl(\frac{(\bar{\mathcal{L}}|_{E_{i}}^{n-1}.(\overline{K_{\mathcal{X}/C}}^{\it Ric(\omega_{h})}))}{(\bar{\mathcal{L}}|_{E_{i}})^{n}}\biggr)^{2}+
\frac{1}{n^{2}}
\sum_{\sigma \in S^{\infty}}\int_{\mathcal{X}(\sigma)}S(\omega_{h})^{2}\omega_{h}^{n}(>0).$$
Here, as above, $\cup_{i}E_{i}={\it Supp}(\mathcal{X}_{0})$ is the support of the special fiber over 
a closed point $0\in C$. 
\end{Def}


\section{Further discussions}\label{sec.3}

In this section, we argue closely related issues as well as giving applications. 

\subsection{Failure of asymptotic semistable reduction}\label{counterexamples}

First, let us recall that, as far as we concern Chow stability of embedded projective varieties, 
we have stable reduction theorem which we recall in the following general form. 

\begin{Prop}[GIT (poly)stable reduction (\cite{Mum77}+\cite{Ses77})]
Suppose $R$ is a discrete valuation ring which is a Nagata ring
\footnote{Noetherian ring which is ``universally Japanese'' (i.e. all finitely generated integral domain extension $(R\subset)R'$ is ``Japanese'' i.e. satisfies the 
finiteness of integral closure in finite extensions $L'/K'$ of the fractional field 
$K'={\it Frac}(R')$)}
 as well and 
let $K$ be its fractional field, $k$ be its residue field. $G$ be a reductive group scheme over $R$ (i.e. 
all of its geometric fibers are reductive algebraic groups) acting on a projective scheme 
$(H,\mathcal{O}_{H}(1))$ over $R$. If $x\in H^{\it ss}(K)$, a semistable point, 
then for a finite extension of $K'$ and the integral closure $R'$ of $R$ in $K'$, 
$x$ extends to a morphism $\tilde{x}\colon{\it Spec}(R')\to H^{\it ss}$ such that 
$x(k)\in H^{\it ss}(k)$ is a polystable\footnote{i.e. semistable with minimal (closed) orbit} point. 
\end{Prop}

\begin{proof}
As the proof for geometric case ($R=k[[t]]$) is written in \cite[Lemma 5.3]{Mum77} 
and basically our general case follows similarly, we only briefly note the difference we need to take 
care. The facts used in the geometric case's proof of Mumford which is not proven for arithmetic case 
in the original G.~I.~T \cite{Mum65} nor \cite{Mum77} were, first, 
the existence of quasi-projective GIT quotient (with its 
compatibility with base change to fibers) and also the existence of group-invariant homogeneous polynomial 
separating arbitrary given group invariant closed subsets. The former is established as 
\cite[p269. Theorem 4, note (v)]{Ses77} and the latter is established as 
\cite[p254. Prop 7 (3)]{Ses77}. 
\end{proof}

However, if we consider the abstract polarized variety $(X,L)$ and consider 
\textit{asymptotic Chow (semi)stability} i.e. the Chow (semi)stability for 
$X\subset \mathbb{P}(H^{0}(X,L^{\otimes m}))$ for $m\gg 0$, then the desired stable 
reduction fails. We give counterexamples below but the key proposition is the following observation 
after the ``local stability'' theory of Eisenbud-Mumford. 

\begin{Prop}[\cite{Mum77}, \cite{Sha81}\footnote{also compare \cite{Oda13b} which 
gives a modern version of them via discrepancy}]\label{mult.unst}
Suppose $(X,L)$ is a $n$-dimensional projective variety over a field. 
If there is a closed point $x\in X$ such that ${\it mult}_{x}(X)> (n+1)!$, 
then $(X,L)$ is asymptotically Chow \textbf{un}stable i.e. 
there for $l\gg 0$, $X\subset \mathbb{P}(H^{0}(X,L^{\otimes l}))$, 
embedded by the complete linear system, is Chow unstable. 
\end{Prop}

\noindent 
The above proposition Proposition \ref{mult.unst} has been used repeatedly, as the key, in 
Shepherd-Barron \cite{She83}, \cite{Oda11a}, Wang-Xu \cite{WX14} etc to show that ``classical GIT does \textit{not}  
work for compactifying moduli of higher dimensional varieties''. 
Although also the Koll\'ar's surface example (\cite{WX14}) essentially works in our situation as well, 
we give a simpler series of examples in arbitrary dimensions as follows. 

\begin{Ex}\label{highmult.ex}
For each prime number $p$, we consider the integer parameters $a_{0},\cdots,a_{n}$ which are 
all coprime to $p$ and coprime to each other. Then we consider an integral model of 
(weighted) Brieskorn-Pham type 
$$(\mathcal{X},\mathcal{O}(1)):=\biggl[\sum_{0\le i< n}x_{i}^{d_{i}}+px_{n}^{a_{0}\cdots a_{n-1}}
=0\subset \mathbb{P}_{\mathbb{Z}}(a_{0},\cdots,a_{n})\biggr],$$
where $d_{i}:=\Pi_{j\neq i}a_{j}$ and ${\it min}_{i}\{a_{i}\}\gg n$. 
The reduction $\mathcal{X}_{p}$ at $p$ is 

$$[\sum_{0\le i< n}x_{i}^{d_{i}}=0\subset \mathbb{P}(a_{0},\cdots,a_{n})].$$

\noindent
Obviously, $\mathcal{X}$ is regular scheme and $\mathcal{X}_{p}$ has only one singular point 
$x=[0:\cdots:0:1]$. The singularity is a quotient of 
$$(0,\cdots,0)\in [\sum_{0\le i\le n-1}X_{i}^{\Pi_{0\le j\le n}a_{j}}]\subset \mathbb{A}^{1}_{X_{0},\cdots,X_{n-1}}$$ (we put $x_{i}=X_{i}^{a_{i}}$) 
by $\Pi_{0\le i\le n}\mathbf{\mu}_{a_{i}}$ thus log canonical. 
In the meantime, the multiplicity of $x\in \mathcal{X}_{p}$ is at least that of cyclic quotient singularity 
$\frac{1}{a_{n}}(a_{0},\cdots,a_{n-1})$. If we choose $a_{0},\cdots,a_{n}$ carefully 
it is easy to make the multiplicity bigger than $(n+1)!$. From the way we take $a_{i}$s, 
all the fibers of $\mathcal{X}$ are normal with ample canonical class. 
\end{Ex}

To show some pathological properties of the above examples, 
we need some preparations. 
First we recall the arithmetic (twisted) variant of Chow weight 
defined by \cite{Bost96}. Also \cite{Bost94},\cite{Zha96} contains 
very closely related variants 
\footnote{But please be a little careful as, the author supposes,  
there are some (unessential) typos which do not cause any troubles, in the definitions 
of their papers: \cite[Theorem I,III]{Bost94} have presumably wrong signs on the right hand sides, 
and in \cite{Zha96}, the place where ``$[K:\mathbb{Q}]$'' is put seem to be wrong. }
and we propose a common generalisation a while later at Definition \ref{tilde.hc}. 

\begin{Def}[\cite{Bost96}]\label{Chow.ht}
Keeping the notation, we suppose further that $\mathcal{X}$ is generically smooth. 
Then the \textit{Chow height} 
$h_{C}(\mathcal{X},\bar{\mathcal{L}}=(\mathcal{L},h))$ is defined as 
$$\dfrac{(\bar{\mathcal{L}})^{n+1}}{({\it dim}(X)+1)(\mathcal{L}_{\eta})^{n}[K:\mathbb{Q}]}
-\dfrac{\hat{\it deg}(\pi_{*}\bar{\mathcal{L}})}{{\it rank}(\pi_{*}\mathcal{L})[K:\mathbb{Q}]},$$
where the direct image sheaf $(\pi_{*}\bar{\mathcal{L}})$ is with the natural $L^{2}$-metric $h_{L^{2}}$ 
which we obtain via $h$ and the K\"ahler metric corresponding to $c_{1}(L,h)$: 

$$h_{L^{2}}(s,\bar{t}):=\int_{\mathcal{X}(\mathbb{C})}<s(x),\overline{t(x)}>_{h} \omega_{h}^{n}.$$

\noindent
Note that the definition of $L^{2}$ metric above is slightly different from the original 
\cite{Bost96} which normalizes the 
volume form to be a probablity measure. This adjustment is 
for the compatibility with our \ref{tilde.hc}, \ref{hc.min}. 
\end{Def}

\begin{Def}\label{rel.ht}
We fix a reference integral model 
$\pi\colon (\mathcal{X}_{\it ref},\mathcal{L}_{\it ref}, h_{\it ref})\to 
{\it Spec}(\mathcal{O}_{K})$ of a generically smooth 
projective variety over the ring of integer of a number field $K$. 
Then for another integral model $(\mathcal{X},\mathcal{L},h_{\it ref})$, 
we set 
$$
h_{K,(\mathcal{X}_{\it ref},\mathcal{L}_{\it ref})}
(\mathcal{X},\mathcal{L}):=h_{K}(\mathcal{X},\mathcal{L},h_{\it ref})-h_{K}(\mathcal{X}_{\it ref},
\mathcal{L}_{\it ref},h_{\it ref}), 
$$
and call it the ``\textit{relative (scheme theoretic) modular height}'' of $(\mathcal{X},\mathcal{L})$ 
with respect to $(\mathcal{X}_{\it ref},\mathcal{L}_{\it ref})$. Note that it easily follows from 
Proposition \ref{ADF.K} that it is 
independent of choice of the reference metric $h_{\it ref}$ so that is well-defined and is a 
quantity of purely arithmetic nature. 

Similarly, we set 
$$h_{C,(\mathcal{X}_{\it ref},\mathcal{L}_{\it ref})}
(\mathcal{X},\mathcal{L}):=h_{C}(\mathcal{X},\mathcal{L},h_{\it ref})-h_{C}(\mathcal{X}_{\it ref},
\mathcal{L}_{\it ref},h_{\it ref}), $$
and call it the ``\textit{relative Chow height}'' of $(\mathcal{X},\mathcal{L})$ 
with respect to $(\mathcal{X}_{\it ref},\mathcal{L}_{\it ref})$. Note that it is again easy to prove to be 
independent of choice of the reference metric $h_{\it ref}$ so that is well-defined and is also a 
quantity of purely arithmetic nature. 

\end{Def}

From the definition it is easy to see the cocycle condition. 
\begin{Prop}\label{hk.coc}
For any three integral models $(\mathcal{X}_{i},\mathcal{L}_{i}) (1\le i\le 3)$ of 
common polarized variety $(X,L)$ over a number field $K$, 
we have $$h_{K,(\mathcal{X}_{1},\mathcal{L}_{1})}(\mathcal{X}_{3},\mathcal{L}_{3})
=h_{K,(\mathcal{X}_{1},\mathcal{L}_{1})}(\mathcal{X}_{2},\mathcal{L}_{2})
+h_{K,(\mathcal{X}_{2},\mathcal{L}_{2})}(\mathcal{X}_{3},\mathcal{L}_{3})$$ and 
$$h_{C,(\mathcal{X}_{1},\mathcal{L}_{1})}(\mathcal{X}_{3},\mathcal{L}_{3})
=h_{C,(\mathcal{X}_{1},\mathcal{L}_{1})}(\mathcal{X}_{2},\mathcal{L}_{2})
+h_{C,(\mathcal{X}_{2},\mathcal{L}_{2})}(\mathcal{X}_{3},\mathcal{L}_{3}).$$ 
\end{Prop}

The following \ref{h.lim.hk} is an Arakelov theoretic analogue of the fact that 
``Donaldson-Futaki invariant is a limit of Chow weights'' (cf., the original definition of \cite{Don02}! 
which is for isotrivial geometric families with $\mathbb{G}_{m}$-action) 
but the proof cannot be obtained as simple immitation and we use 
asymptotic analysis of the \textit{Ray-Singer analytic torsion} \cite{RS73} (as well as a 
simple ``anomaly'' formula) in addition to 
the Gillet-Soul\'e arithmetic Riemann-Roch \cite{GS92}. 

\begin{Thm}[(De-)Quantization]\label{h.lim.hk}
Keeping the notation, we still suppose that $\mathcal{X}$ is generically smooth over 
$C={\it Spec}(\mathcal{O}_{K})$. Then the following asymptotic behavior of Chow heights holds: 
$$ 
h_{C}(\mathcal{X},\mathcal{L}^{\otimes m},h^{m})=2(n+1)(L^{n})^{2}
h_{K}(\mathcal{X},\mathcal{L},h)+\frac{n}{4}{\it log}(m)+o(1), 
$$
\noindent
for $m\to \infty$. 
Hence, in particular, the (K-)modular height is essentially a limit of slightly modified Chow heights: 
$$h_{K}(\mathcal{X},\mathcal{L},h)=\frac{1}{2(n+1)(L^{n})^{2}}
\biggl(h_{C}(\mathcal{X},\mathcal{L}^{\otimes m},h^{m})-\frac{n}{4}{\it log}(m)\biggr)+o(1)$$ 
\noindent 
for $m\to \infty$. 
\end{Thm}

Due to the presence of the logarithmic term ($\frac{n}{4}{\it log}(m)$), 
it also shows that Zhang's height positivity conjecture \cite[p78]{Zha96} is 
always ``asymptotically true" with respect to the twist of line bundle, 
even without the (Chow) semistability assumption. 
The author is happy to acknowledge that 
Robert Berman told me in May of 2016 that he had a closely related result to 
\ref{h.lim.hk} in 2012. 

\begin{proof}[First proof]
Clearing the denominators of the Chow heights $h_{C}(\mathcal{X},\mathcal{L}^{\otimes m},h^{m})$, 
what we are to analyse is the asymptotic behaviour of 
$$(\bar{\mathcal{L}}^{n+1})h^{0}(L^{\otimes m})m-(n+1)(L^{n})\hat{\it deg}(\pi_{*}\bar{\mathcal{L}}^{\otimes m})$$
$$=m(\bar{\mathcal{L}}^{n+1})\biggl(\frac{(L^{n})}{n!}m^{n}-\frac{(L^{n-1}.K_{X})}{2(n-1)!}m^{n-1}+\cdots 
\biggr)-(n+1)(L^{n})\hat{\it deg}(\pi_{*}\bar{\mathcal{L}}^{\otimes m}, h_{L^{2}})$$ 
with respect to $m\gg 0$. To clarify that the $L^{2}$-metric above is induced by 
$h^{m}$ on $L^{\otimes m}$ and $mg_{h}$ on $T_{X}$ corresponding to $c_{1}(L^{\otimes m},h^{m})$, 
is should be read as $h_{L^{2}}(h^{m},mg_{h})$. Note that this is \textit{not} 
$h_{L^{2}}(h^{m},g_{h})$. So we make the asymptotic analysis of 
$\hat{\it deg}(\pi_{*}\bar{\mathcal{L}}^{\otimes m},h_{L^{2}}(h^{m},mg_{h}))$. We decompose it as 
$$\hat{\it deg}(\pi_{*}\bar{\mathcal{L}}^{\otimes m},h_{L^{2}}(h^{m},g_{h}))$$
$$+\bigl(\hat{\it deg}(\pi_{*}\bar{\mathcal{L}}^{\otimes m},h_{L^{2}}(h^{m},mg_{h}))
-\hat{\it deg}(\pi_{*}\bar{\mathcal{L}}^{\otimes m},h_{L^{2}}(h^{m},g_{h}))\bigr)$$
and then 

\begin{equation}\label{aa}
=\bigl(\hat{\it deg}(\pi_{*}\bar{\mathcal{L}}^{\otimes m},h_{L^{2}}(h^{m},mg_{h}))
-\hat{\it deg}(\pi_{*}\bar{\mathcal{L}}^{\otimes m},h_{L^{2}}(h^{m},g_{h}))\bigr), 
\end{equation}
\begin{equation}\label{bb}
+\bigl(\hat{\it deg}(\pi_{*}\bar{\mathcal{L}}^{\otimes m},h_{L^{2}}(h^{m},g_{h}))
-\hat{\it deg}(\pi_{*}\bar{\mathcal{L}}^{\otimes m},h_{Q}(h^{m},g_{h}))\bigr)
\end{equation}
\begin{equation}\label{cc}
+\hat{\it deg}(\pi_{*}\bar{\mathcal{L}}^{\otimes m},h_{Q}(h^{m},g_{h}))
\end{equation}

\noindent
where $h_{Q}$ of (\ref{bb}),(\ref{cc}) stand for the Quillen metrics. 
It follows directly from the definition that (\ref{aa}) coincides with $\frac{[K:\mathbb{Q}]}{2}$ times 
$${\it rank}(\pi_{*}\mathcal{L}^{\otimes m})\cdot {\it log}(m^{n})=
\frac{(L^{n})}{(n-1)!}m^{n}{\it log}(m)+O(m^{n-1}{\it log}(m)).$$ Note that 
${\it log}(m^{n})$ appears as an entropy. From the definition of the Quillen metric, 
the second part (\ref{bb}) is half of the Ray-Singer analytic torsion which we denote as 
$$\frac{d\zeta ^{\it sp,g}_{m}(s)}{ds}|_{s=0}=\sum_{q=0}^{n}(-1)^{q+1}q\frac{d\zeta^{\it sp,g}_{m,q}(s)}{ds}|_{s=0}.$$ 
We also denote the above as 
$T(X_{\infty},g,L_{\infty}^{\otimes m},h^{m})$. 
\footnote{Different notation from \cite{Bost96} up to additive constant due to 
normalisation of K\"ahler form and the difference of $L^{2}$ metric} 
Here, note that $\zeta_{m,q}^{\it sp,g}(s)$ denotes the spectral zeta function 
$\sum_{n} \lambda_{m,q,n}(g)^{-s}$, the sum of $(-s)$-powers of all positive eigenvalues 
$\lambda_{m,q,n}(g)$ of $\bar{\partial}$-Laplacian 
$\Delta_{\bar{\partial},m}(g)$ on $\mathcal{A}^{0,q}(L^{\otimes m})$, the space of 
$(0,q)$-forms of $C^{\infty}$ class with coefficients in $L^{\otimes m}$. We set 
$\zeta_{m}^{\it sp,g}(s):=\sum_{q}(-1)^{q+1}q\zeta_{m,q}^{\it sp,g}(s)$. 
In our situation, $g=g_{h}$. 
By Bismut-Vasserot (\cite[Theorem 8]{BV89}),\footnote{Note that $r^{\circ}/2\pi$ of the paper is 
the identity matrix ${\it Id}$ in 
our situation} we conclude that (\ref{bb}) is 
$$\frac{1}{2}T(X_{\infty},g,L_{\infty}^{\otimes m},h^{m})=
\frac{[K:\mathbb{Q}]}{4}\biggl\{\frac{(L^{n})}{(n-1)!}m^{n}{\it log}(m)+o(m^{n})\biggr\}.$$ 

Finally we have that (\ref{cc}) is 
$$\hat{\it deg}(\pi_{*}\bar{\mathcal{L}}^{\otimes m},h_{Q}(h^{m},g_{h}))=
\hat{\it deg}(\pi_{*}\bar{\mathcal{L}}^{\otimes m},h_{Q}(h^{m},mg_{h}))$$
$$=\bar{{\it deg}}\biggl((m^{n+1}\frac{\hat{c_{1}}(\bar{\mathcal{L}})^{n+1}}{(n+1)!}
+m^{n}\frac{\hat{c_{1}}(\bar{\mathcal{L}})^{n}}{n!}+o(m^{n}))\cdot \hat{\it td}(T_{\mathcal{X}/C}, 
\omega_{h})\cdot (1-a(R(T_{X(\mathbb{C})},\omega_{h}))\biggr)$$
$$
=\frac{\hat{c_{1}}(\bar{\mathcal{L}})^{n+1}}{(n+1)!}m^{n+1}+
\frac{(\hat{c_{1}}(\bar{\mathcal{L}})^{n}.\hat{c_{1}}(T_{\mathcal{X}/C},\omega_{h}))}{2(n!)}m^{n}+o(m^{n}), 
$$
by the Gillet-Soul\'e arithmetic Riemann-Roch theorem \cite{GS92}, 
where ``$a(R(-,-))$'' is the R-genus. 

Combining the above analysis, we have the following fundamental 
refinement of asymptotic Hilbert-Samuel formula (originally Gillet-Soul\'e \cite{GS88}). 
The reason the author gives its proof and statement here is that he unfortunately 
could not find any literature for it, but he presumes it is well-known to experts. 

\begin{Prop}[Asymptotic Hilbert-Samuel formula]\label{aHS}
Suppose that 
$(\mathcal{X},\mathcal{L})$ is a normal polarized projective variety over 
${\it Spec}(\mathcal{O}_{K})$, which is furthermore generically smooth, 
and a hermitian metric $h$ on $L_{\infty}$ of real type. Then we have, 
for $m\to \infty$, 
$$\hat{\it deg}(\pi_{*}\bar{\mathcal{L}}^{\otimes m},h_{L^{2}}(h^{m},mg_{h}))
=$$
$$\frac{(\bar{\mathcal{L}}^{n+1})}{(n+1)!}m^{n+1}-\frac{(L^{n})}{4((n-1)!)}m^{n}{\it log}(m)
-\frac{(\bar{\mathcal{L}}^{n}.\overline{K_{\mathcal{X}/C}})}{2(n!)}m^{n}+o(m^{n}).$$
\end{Prop}

From the above \ref{aHS}, we have 
$$(\bar{\mathcal{L}}^{n+1})h^{0}(L^{\otimes m})m^{n+1}-(n+1)(L^{n})
\biggl(\hat{\it deg}(\pi_{*}\bar{\mathcal{L}}^{\otimes m})-\frac{(L^{n})}{4((n-1)!)}m^{n}{\it log}(m)\biggr)$$
$$=\frac{1}{2(n!)}h_{K}(\mathcal{X},\mathcal{L},h)m^{n+1}+O(m^{n}), $$
which completes the first proof of \ref{h.lim.hk}. 
\end{proof}

\begin{proof}[More direct analytic proof of Theorem \ref{h.lim.hk}]
We can prove the above theorem without (really) concerning metrics of type $h_{**}(h^{m},g_{h})$. 
The asymptotic analysis of $\hat{\it deg}(\pi_{*}\bar{\mathcal{L}}^{\otimes m},h_{L^{2}}(h^{m},mg_{h}))$ 
can be replaced as: 

$$
\hat{\it deg}(\pi_{*}\bar{\mathcal{L}}^{\otimes m},h_{L^{2}}(h^{m},mg_{h}))
$$
\begin{equation}\label{dd}
=\bigl(\hat{\it deg}(\pi_{*}\bar{\mathcal{L}}^{\otimes m},h_{L^{2}}(h^{m},mg_{h})) 
-\hat{\it deg}(\pi_{*}\bar{\mathcal{L}}^{\otimes m},h_{Q}(h^{m},mg_{h}))\bigr)
\end{equation}
\begin{equation}\label{ee}
+\hat{\it deg}(\pi_{*}\bar{\mathcal{L}}^{\otimes m},h_{Q}(h^{m},mg_{h}))\bigr), 
\end{equation}

Note that (\ref{dd}) is half of the Ray-Singer torsion 
$T(X,mg,L^{\otimes m},h^{m})$. From the anomaly formula (\cite[Proposition 4.4]{Bost96}) 
combined with asymptotics of the Ray-Singer torsion \cite[Theorem 8]{BV89}, 
we have 

\begin{Lem}\label{as.a.t}
$T(X_{\infty},mg_{h},\bar{L_{\infty}}^{\otimes m})=o(m^{n}). $
\end{Lem}

\noindent
This lemma \ref{as.a.t} extends  \cite[Proposition 4.2]{Bost96} for abelian varieties in a weak form 
(but note the slight difference with \cite{Bost96} due to normalisations in the definition). 
On the other hand, (\ref{ee}) can be calculated by the Gillet-Soul\'e arithmetic Riemann-Roch theorem \cite{GS92} again, 
as (\ref{ee}) equals to 
$$\hat{{\it deg}}\biggl((m^{n+1}\frac{\hat{c_{1}}(\bar{\mathcal{L}})^{n+1}}{(n+1)!}
+m^{n}\frac{\hat{c_{1}}(\bar{\mathcal{L}})^{n}}{n!}+o(m^{n}))\cdot \hat{\it td}(T_{\mathcal{X}/C}, 
m\omega_{h})\cdot (1-a(R(T_{X(\mathbb{C})},\omega_{h}))\biggr), $$
where ``$a(R(-,-))$'' is the R-genus again. If we use 
the description of secondary class \cite[(4.2.9)]{Bost96}, 
we have 
$$\hat{\it td}(T_{\mathcal{X}/C}, 
m\omega_{h})=\hat{\it td}(T_{\mathcal{X}/C}, 
\omega_{h})-{\it log}(m)(Td'(T_{X_{\infty}},\omega_{g_{h}})),$$

\noindent
where $Td'$ stands for the characteristic form defined by the derivative of 
formal series which corresponds to the Todd class (cf., \cite[4.2.2]{Bost96}). 

As being similar to the first proof, simple additions of the above gives us the assertion again. 
\end{proof}

We have the following application of the above de-quantisation process 
(\ref{h.lim.hk}). 
This is an arithmetic version of \cite{WX14}, which in turn is builds upon 
\cite[3.12]{Mum77},\cite[1.1 and its proof]{Oda11a} (cf., also \cite[section 3]{Oda12}). 

\begin{Thm}
The generic fiber $(\mathcal{X}_{\eta},K_{\mathcal{X}_{\eta}})$ of the example (\ref{highmult.ex}) 
does \underline{not} have weakly asymptotic Chow semistable reduction at the prime $p$ i.e. 
there is \underline{no} integral model $(\mathcal{X}',\mathcal{L}')$ which satisfies that it 
extends $(\mathcal{X}_{\eta},K_{\mathcal{X}_{\eta}})$ and for infinitely many $l\gg 0$, 
 $(\mathcal{X}'_{p},\mathcal{L}'_{p})$ embedded by $|(\mathcal{L}'_{p})^{\otimes l}|$ are 
 Chow semistable. 
\end{Thm}

\begin{proof}
Suppose the contrary and let $\pi'\colon(\mathcal{X}',\mathcal{L}')\to C$ 
be such a integral model with weakly asymptotically 
Chow semistable reduction $(\mathcal{X}'_{p},\mathcal{L}'_{p})$ at $p$. 
We also fix a reference integral model 
$\pi\colon (\mathcal{X}_{\it ref},\mathcal{L}_{\it ref}, h_{\it ref})\to 
{\it Spec}(\mathcal{O}_{K})$ as above. 
We follow the strategy of \cite{WX14} and deduce contradiction by using Theorem \ref{mult.unst}. 

From the assumption of the contrary, 
there are infinitely many $m\gg0$ such that 
$[\mathcal{X}'\subset \mathbb{P}(\pi'_{*}\mathcal{L}'^{\otimes m}))]$ is 
minimising the relative Chow height $h_{C,(\mathcal{X}_{\it ref},\mathcal{L}_{\it ref}^{\otimes m})}$ 
(\ref{Chow.ht},\ref{rel.ht})
among all integral models. Thus, from Proposition \ref{h.lim.hk}, 
it also minimizes $h_{K,(\mathcal{X}_{\it ref},\mathcal{L}_{\it ref})}$ (\ref{rel.ht}) or equivalently 
minimising the modular height $h_{K}$ (if we fix reference metric). 
Thus, Theorem \ref{hk.min} tells us that $(\mathcal{X}',\mathcal{L}')\simeq (\mathcal{X},\mathcal{L})$. 

On the other hand, Proposition \ref{mult.unst} shows that the multiplicity of any closed point in 
$\mathcal{X}_{p}$ at most $(n+1)!$. This contradicts to the fact that 
our example \ref{highmult.ex} violates the condition. 
\end{proof}

\begin{Rem}
In our set of examples \ref{highmult.ex}, if we replace ${\it Spec}(\mathbb{Z})$ and $p$ with 
$\mathbb{A}_{k}^{1}={\it Spec}(k[t])$ and $t$ for an algebraically closed field $k$, 
they give yet another but simpler variants to the examples of 
\cite{She83},\cite{Oda11a},\cite{WX14}. 
\end{Rem}

\subsection{Arithmetic Yau-Tian-Donaldson conjecture}\label{ar.YTD.sec}

We propose the following conjectures which speculate relations between purely metrical properties 
of (arithmetic) varieties and their purely arithmetic properties. 

\begin{Conj}[Arithmetic Yau-Tian-Donaldson conjecture]\label{arith.YTD}
For an arbitrary smooth projective variety $(X,L)$ with finite automorphism group 
${\it Aut}_{K}(X,L)$ over a number field $K$, the following conditions are equivalent. 
\begin{enumerate}

\item \label{1} (Differential geometric side) $X(\mathbb{C})$ admits a constant scalar curvature K\"ahler metric 
$\omega\in c_{1}(L(\mathbb{C}))$.

\item \label{2} (Arithmetic side) 
There is an integral model $(\mathcal{X},\mathcal{L})$ 
of $(X,L)$ possibly after finite extension of $K$, such that 
for each prime ideal $\mathfrak{p}$ of $\mathcal{O}_{K}$, 
the reduction $(\mathcal{X}_{\mathfrak{p}},\mathcal{L}_{\mathfrak{p}})$ is 
K-semistable. Furthermore, for almost all (other than finite exceptions) $\mathfrak{p}$, 
the reduction is K-stable. 

Note that this property is purely arithmetic 
(the notion does not depend on reference model a posteriori). 
$(X,L)$ is said to be \textbf{arithmetically K-stable} if ${\it Aut}_{K}(X,L)$ is 
finite and the above conditions (the previous paragraph of this (2)) hold. 
\end{enumerate}
\end{Conj}

We believe that the above equivalence also holds for singular varieties (cf., e.g., \cite
{EGZ09}) 
and log pairs (cf., e.g., \cite{Don12},\cite{OS15}) as well. 
For example, if the (log-)canonical class is ample or numerically trivial, it is natural to 
admit geometrically \textit{semi-log-canonical} singularities and if the (log-)canonical class is anti-ample, it is 
natural to admit \textit{(kawamata-)log terminal} singularities 
(cf., \cite{Oda11a}, \cite{Oda13b},\cite{BG14},\cite{OS15}).

We call the lower boundedness of modular height 
$h_{K}$ among all (metrized) integral models $(\mathcal{X},\mathcal{L},h)$ 
over finite extensions of $K$, the \textbf{Arakelov K-semistability} of $(X,L)$. 
Then the semistability version of the above conjecture is as follows.

\begin{Conj}[Semistable arithmetic YTD conjecture]\label{arith.YTD.ss}
For an arbitrary smooth projective variety $(X,L)$ 
over a number field $K$, the followings are equivalent. 
\begin{enumerate}

\item \label{1s} (Differential geometric side) $(X(\mathbb{C}),L(\mathbb{C}))$ is K-semistable 
(in \cite{Don02} sense). 

\item \label{2s} (Mixed) $(X,L)$ is Arakelov K-semistable i.e., 
$h_{K}(\mathcal{X},\mathcal{L},h)$, where $(\mathcal{X},\mathcal{L})$ runs through 
all the models of $(X,L)$ and all $h$ whose curvature is positive, are 
(uniformly) lower bounded.

\item \label{3s} (Arithmetic side) For an integral model $(\mathcal{X},\mathcal{L})$ over $\mathcal{O}_{K}$, 
and a maximal ideal $\mathfrak{p}\in {\it Spec}(\mathcal{O}_{K})$, the reduction 
$(\mathcal{X}_{\mathfrak{p}},\mathcal{L}_{\mathfrak{p}})$ is K-semistable. 
\footnote{Note that the definition of K-stability \cite{Don02} works over any field, although 
in the paper the base is assumed to be $\mathbb{C}$.}

\item \label{4s} (Arithmetic side) For any integral model $(\mathcal{X},\mathcal{L})$ over $\mathcal{O}_{K}$,  
for almost all (i.e. outside finite exceptions) maximal ideals $\mathfrak{p}\in {\it Spec}(\mathcal{O}_{K})$, the reduction 
$(\mathcal{X}_{\mathfrak{p}},\mathcal{L}_{\mathfrak{p}})$ is K-semistable. 
We call this condition of $(X,L)$ \textbf{arithmetic K-semistability}. 
\end{enumerate}
\end{Conj}

The above expected equivalence between (\ref{1s}) and  (\ref{3s}),(\ref{4s}), when it is true, 
give a partial geometric meaning of 
K-(semi)stability of polarized varieties over \textit{positive} characteristic field via lifting. 

We partially confirm the above conjecture as follows. 
First, it straightforwardly follows from Proposition (\ref{ADF.K}) that (\ref{2s}) of the above 
conjecture \ref{arith.YTD.ss} implies lower boundedness of the 
Mabuchi's K-energy (cf., \cite{Li13} for (geometric) Fano case). 
The geometric version of the above conjecture is (partially) discussed in \cite[section3, subsection 4.2]{Oda15a}), where the Arakelov K-semistability corresponds to ``generic K-semistability'' introduced as 
\cite[Definition 3.1]{Oda15a}. 

A possibly interesting remark would be that the above two conjectures also give us an insight that 
existence or non-existence of K\"ahler-Einstein metrics and other canonical K\"ahler metric 
remains the same once we replace the coefficients of the defining equations by some conjugates 
(i.e. acted by an element of ${\it Gal}(\mathbb{Q})$). 

We only partially confirm the above conjectures as follows. 
As in Theorem \ref{hk.min}, the statements are involving the language of log pairs 
$(\mathcal{X},\mathcal{X}_{\mathfrak{p}})$ and their log discrepancies 
(cf., e.g., \cite{KM98}) but 
it is due to lack of inversion of adjunction which is natural to expect. 
Each of the conditions on $(\mathcal{X},\mathcal{X}_{\mathfrak{p}})$ below morally means mildness of 
singularities of fibers $\mathcal{X}_{\mathfrak{p}}$ 
(or their anti-canonical divisors in the case (3)), 
as it was precisely so in the geometric case. 
We refer to the intersted readers to the algebro-geometric theorems 
\cite[Theorem 6]{WX14}, \cite[section 4]{Oda12} corresponding to the following. 

\begin{Thm}
The following types of projective varieties $(X,L)$ over number field $K$ are arithmetically K-stable 
(resp., arithmetically K-semistable and Arakelov K-semistable), 
and all the base changes to complex places 
admit possibly singular 
K\"ahler-Einstein metrics (resp., K-energy of geometric generic fiber is bounded). 
In particular, the arithmetic Yau-Tian-Donaldson conjecture \ref{arith.YTD},\ref{arith.YTD.ss} hold 
for the following cases. 
\begin{enumerate}
\item Log terminal polarized variety $(X,L)$ over $K$ with numerically trivial $K_{X}$ 
which admits, possibly after finite extension of $K$, an integral model $(\mathcal{X},\mathcal{L})$ 
(``relatively minimal model'') that satisfies the following. 
$K_{\mathcal{X}}$ is $\mathbb{Q}$-Cartier and for any prime $\mathfrak{p}$ of $\mathcal{O}_{K}$, 
the reduction $\mathcal{X}_{\mathfrak{p}}$ is reduced with $\mathbb{Q}$-linearly 
trivial canonical divisor and 
$(\mathcal{X},\mathcal{X}_{\mathfrak{p}})$ is log canonical in an open neighborhood of 
$\mathcal{X}_{\mathfrak{p}}$. 

\item Log terminal variety \footnote{In this case and case (1), 
the only technical reason why we do not loose this 
mildness assumption of singularities to (semi-)log canonicity, 
which is more natural, is due the difficulty of the metric \cite{BG14}. Same for (3). } $X$ over $K$ 
with ample (pluri-)canonical polarization 
$L=\mathcal{O}_{X}(mK_{X})$ with $m\in \mathbb{Z}_{>0}$, which has 
a nice (``relative log canonical model'') integral model $\mathcal{X}$, possibly after finite extension of $K$, 
which satisfies the following. 
$\mathcal{X}$ is $\mathbb{Q}$-Cartier and 
each reduction $\mathcal{X}_{\mathfrak{p}}$ is reduced with ample canonical $\mathbb{Q}$-Cartier divisor 
and $(\mathcal{X},\mathcal{X}_{\mathfrak{p}})$ is log canonical in a neighborhood of 
$\mathcal{X}_{\mathfrak{p}}$. 

\item 
Log terminal (anti-canonically polarized) $\mathbb{Q}$-Fano variety $(X,L)$ over $K$ 
which has, after 
some finite extension of $K$ if necessary, an integral model $\mathcal{X}$ which satisfies the 
following. 
For any prime $\mathfrak{p}$ of $\mathcal{O}_{K}$, 
$\mathcal{X}_{\mathfrak{p}}$ is a 
(normal and) log terminal $\mathbb{Q}$-Fano variety 
whose alpha invariant 
$${\it sup}\{t\ge 0 \mid (\mathcal{X},\mathcal{X}_{\mathfrak{p}}+tD) \text{ is lc around 
$\mathcal{X}_{\mathfrak{p}}$ for all effective } 
D\equiv_{/C} -K_{\mathcal{X}/C}\}$$
(cf., \cite{Tia87a}, \cite{OS12}) is more than $\frac{n}{n+1}$ 
(resp., at least $\frac{n}{n+1}$). Here, ``lc'' stands for log canonical. 
\end{enumerate}
\end{Thm}

\begin{proof}
It simply follows from our height minimisation theorem \ref{hk.min}, 
when we combine with known analytic results that K\"ahler-Einstein metrics exist on 
each geometric generic fibers for the above situations 
\cite{EGZ09},\cite{EGZ11},\cite{Tia87a} (resp., the lower boundedness of K-energy for the 
case of (3) with the alpha invariant $\frac{n}{n+1}$ cf., e.g., \cite{Li13}). 
\end{proof}

\noindent
Note that the desired integral models (1),(2) are naturally expected to exist unconditionally 
as consequences of the conjectural arithmetic Log Minimal Model Program. 

We now review the original ``quantised'' version i.e. the version for 
Chow stability of (essentially) \textit{embedded} varieties with slight improvement. That is, 
we slightly extend the definition of Chow height 
(\cite{Bost94},\cite{Bost96},\cite{Zha96}) 
after an idea of Donaldson \cite{Don04}. 

\begin{Def}[Extended Chow height]\label{tilde.hc}
We keep the notation as above (although $\mathcal{X}$ only need to be normal and 
$\mathbb{Q}$-Gorenstein. )
As extra data, a hermitian metric of real type\footnote{i.e. complex conjugate invariant} 
$h$ on $L=\mathcal{L}(\mathbb{C})$ and a hermitian metric of real type $H$ on $\pi_{*}\mathcal{L}$, 
we associate the 
\textit{extended Chow height} $\tilde{h}_{C}$ as 

$$
\tilde{h}_{C}(\mathcal{X},\mathcal{L},h,H):=
\frac{1}{[K:\mathbb{Q}]}\times
$$
$$ \biggl\{
\dfrac{(\bar{\mathcal{L}})^{n+1}}{({\it dim}(X)+1)(\mathcal{L}_{\eta})^{n}}
-\dfrac{\hat{\it deg}(\pi_{*}\bar{\mathcal{L}})}{{\it rank}(\pi_{*}\mathcal{L})}
+{\it log}\biggl(
\int_{X_{\infty}}\sum_{\alpha}|s_{\alpha}|_{h}^{2}\frac{c_{1}(L,h)^{n}}{(L^n)}\biggr)
\biggr\}, 
$$
where $s_{\alpha}$ is the orthonomal basis of $H^{0}(L)$ with respect to $H$. 
Note that the integrand $\sum_{\alpha}|s_{\alpha}|_{h}^{2}$ is the ratio $\frac{h}{\it FS(H)}$. 
The difference with the original Chow height ($h_{C}$ in this paper) 
introduced in \cite{Bost96}, \cite{Zha96} is the last 
logarithmic term. 
\end{Def}
The following equivalence can be nowadays regarded 
as a quantised version of the Yau-Tian-Donaldson correspondence, 
slightly refined after \cite{Don04}. One direction (implication from semi-stability) 
was found by J-B.Bost \cite{Bost94},\cite{Bost96} and S.Zhang \cite{Zha96} while 
the other direction is also found by Zhang \cite{Zha96}. We slightly 
refine the statement partially after \cite{Don04}. 

\begin{Thm}[Quantized Yau-Tian-Donaldson correspondence]\label{hc.min}
If we fix a smooth polarized projective variety 
$(X,L)$ over a number field $K$, its Chow semistability (resp., Chow polystability) 
is equivalent to lower boundedness (resp., existence of minimum) of 
$\tilde{h}_{C}(\mathcal{X},\mathcal{L},h,H)$ where $(\mathcal{X},\mathcal{L})$ is an integral 
model of $(X,L)$ with 
a real type hermitian metric $h$ on $L(\mathbb{C})$ and a real type hermitian metric (inner product) 
$H$ on $H^{0}(L(\mathbb{C}))$. 

Indeed, if $(X,L)$ is Chow polystable, it 
minimizes at the model with the balanced metric $h$, its corresponding 
$L^{2}$-metric $H$, and Chow polystable reductions at all primes. 
\end{Thm}

\begin{proof}
Note the last term 
${\it log}\bigl(
\int_{X_{\infty}}\sum_{\alpha}|s_{\alpha}|_{h}^{2}c_{1}(L,h)^{n}\bigr)
$ of $\tilde{h}_{C}$ only depend on the archimedean data. So what we first want to know is that once we fix 
a reference model $(\mathcal{X}_{\it ref},\mathcal{L}_{\it ref},h_{\it ref})$ (with some fixed $H$), 
the relative Chow height $h_{C,(\mathcal{X}_{\it ref},\mathcal{L}_{\it ref})}
(\mathcal{X},\mathcal{L})$ (Definition \ref{rel.ht}) minimizes exactly when 
$(\mathcal{X}_{\mathfrak{p}},\mathcal{L}_{\mathfrak{p}})$ is Chow polystable 
for any maximal ideal $\mathfrak{p}$ of $\mathcal{O}_{K}$. 
This is proved by Zhang \cite{Zha96}. Semistable version of the statement is 
also established by \cite{Bost94},\cite{Bost96} for one direction and 
\cite{Zha96} for both directions. 

Thus what remains to show is that when we fix the scheme-theoretic data $(\mathcal{X},\mathcal{L})$, 
$\tilde{h}_{C}(\mathcal{X},\mathcal{L},h,H)$ minimizes at balanced 
\footnote{originally called ``critical metric'' in Zhang \cite{Zha96}} Fubini-Study metric $h$ and 
its corresponding $H$, i.e. its $L^{2}$-metric. 
Such result 
is essentially proved in Donaldson \cite[Theorem 2, Lemma 4,5]{Don04} where he wrote $\tilde{P}$ for the 
corresponding archimedean invariant. Note that \cite[Lemma4,5]{Don04} also are essentially the same as 
\cite[Proposition 2.1,2.2]{Bost96} while the definition of the $L^{2}$ metric in \cite[(1.2.3)]{Bost96} 
is without the multiplication by ${\it rank}(\pi_{*}\mathcal{L})$. 
We refer to the papers \cite{Bost96},\cite{Zha96} and \cite{Don04} for the details. 
\end{proof}

Going back to our de-quantised version, 
we show some examples of Arakelov K-\textbf{un}stable (i.e. not 
Arakelov K-\textit{semi}stable) 
arithmetic varieties. 

\begin{Ex}\label{1pt.bl.up}
For $C={\it Spec}(\mathbb{Z})$ and the projective plane $\mathcal{Y}:=\mathbb{P}^{2}_{\mathbb{Z}}$ over $C$, 
polarized by $\mathcal{M}:=\mathcal{O}_{\mathcal{X}}(1)$, we can take a section $S\subset \mathcal{X}$. 
Then we set $\mathcal{X}:={\it Bl}_{S}(\mathcal{Y}), \mathcal{L}:=
\mathcal{O}_{\mathcal{X}}(-K_{\mathcal{X}/C})$. If we denote 
by $E$ the exceptional divisor in $\mathcal{X}$, for an arbitrary set of prime numbers $p_{1}<\cdots<p_{m}$, we can construct an integral model $\mathcal{X}(p_{1},\cdots,p_{m})$ as the blow up of 
$\mathcal{X}$ along $\cup_{i}(E|_{\mathcal{X}_{p_{i}}})$ (with its reduced structure) 
and set $\mathcal{L}(p_{1},\cdots,p_{m}):=\pi^{*}\mathcal{L}(-\sum_{i}F_{i})$, 
where $\pi\colon \mathcal{X}(p_{1},\cdots,p_{m})\to \mathcal{X}$ is the blow up and 
$F_{i}$ is the $\pi$-exceptional divisor over the prime $p_{i}$. 

Then we have $h_{K}(\mathcal{X}(p_{1},\cdots,p_{m}),\mathcal{L}(p_{1},\cdots,p_{m}),h)
=h_{K}(\mathcal{X},\mathcal{L},h)-c\sum_{i}{\it log}(p_{i})$ with some positive constant 
$c$, 
thus Arakelov K-unstability of the generic fiber of $(\mathcal{X},\mathcal{L})$ follows. Its 
arithmetic K-unstability also follows the above same arguments. 
\end{Ex}

We end this section by supplementarily giving a definition of ``intrinsic version''after 
the idea of \cite{Bost96} and proposal of a problem about effects of arithmetic structures. 

\begin{Def}[Intrinsic (K-)modular height]
For arithmetically K-semistable arithmetic variety $(X,L)$ over a number field $K$, we set 
$$h_{K}(X,L):={\it inf}_{(\mathcal{X},\mathcal{L},h,K'/K)}h_{K}(\mathcal{X},\mathcal{L},h)$$ where 
$(\mathcal{X},\mathcal{L},h,K')$ run over the set of all (vertically positive) 
metrized models of $(X,L)$ over 
finite extensions $K'$ of $K$, as in Conventions \ref{Not.Cov} (especially cf., (4)).
We 
call it the \textit{intrinsic (K-)modular height}. 

Recall that 
K of ``K-''modular and that of subscript of $h_{K}$ comes from K-stability, hence from K\"ahler after all, 
but not from the original base 
field. 
Indeed, we take all metrized models over all finite extensions of $K$ for the definitions of our 
modular heights. 
\end{Def}
\noindent
This formulation also follows the definition of the Faltings height \cite{Fal83} for 
abelian varieties (recall section \ref{sec.2}, especially \ref{Falt.ht}). 
The name is after \cite[1.2.3]{Bost96}, a quantised (embedded varieties) original version 
of the above, which the author would like to distinguish by calling it \textit{intrinsic Chow height} 
(in our papers) to avoid confusion from now on.

\begin{Rem}[Arithmetic dance]
We can rather fix only complex variety structure as follows. 
Starting with polarized \textit{complex} projective variety $(X,L)$, i.e. with only its $\mathbb{C}$-structure, 
we can consider the set 
$$
\{(K,(\mathcal{X},\mathcal{L},h))\mid c_{1}(L,h)>0, 
(X,L)\text{ is a component of } (\mathcal{X}(\mathbb{C}),\mathcal{L}(\mathbb{C}))\},$$
which we denote by $\mathcal{D}(X,L)$. The point is that we change $K$-structure (arithmetic structure). 
A possible question one can ask is 
``\textit{what is the precise relation among $${\it inf}_{(\mathcal{X},\mathcal{L},h)\in \mathcal{D}(X,L)}h_{K}(\mathcal{X},\mathcal{L},h),$$ the Donaldson-Futaki invariant and K-energy?}''
\end{Rem}

\subsection{Arithmetic moduli}\label{a.K.mod}

We conclude this paper by a brief introduction to our ongoing 
attempt to partially unify treatments of (arithmetic) moduli, 
which we hope to discuss details in the near future. 
First we introduce the following key arithmetic line bundle(s): 

\begin{Def}[Arithmetic line bundles]\label{aa.line}

Suppose $\pi\colon \mathcal{X}\to S$ is a smooth projective morphism 
between quasi-projective normal $\mathbb{Q}$-Gorenstein schemes over $\mathbb{Z}$. 
Furthermore, $\mathcal{L}$ is an arithmetic line bundle on $\mathcal{X}$ which is 
$\pi$-ample with a family of hermitian metrics $h=\{h_{\bar{s}}\}_{\bar{s}}$ 
on $\mathcal{L}|_{\mathcal{X}_{\bar{s}}}$ 
over any geometric points $\bar{s}$ of $S$ are smooth of positive curvature 
(cf., Notations and Conventions \ref{Not.Cov} and \cite{FS90}). 

Then we define the following arithmetic line bundles:  

\begin{enumerate}\label{a.line}

\item (Arithmetic CM line bundle) 
We denote the following arithmetic line bundle on $S$ 
as $\bar{\lambda}_{\it CM}(\mathcal{X},\bar{\mathcal{L}})$ 
and call it the \textit{arithmetic CM line bundle}:  
$$<\bar{\mathcal{L}}^{h},\cdots,\bar{\mathcal{L}}^{h}>^
{\otimes (-(n-1)(L|_{X_{s}}^{n-1}.K_{X_{s}}))}\otimes 
<\bar{\mathcal{L}}^{h},\cdots,\bar{\mathcal{L}}^{h},\overline{K_{\mathcal{X}/S}}^{{\it det}(\omega_h)}>
^{\otimes ((n+1)(L|_{X_{s}}^{n}))},$$ 
where $<\cdots>$ denotes the Deligne pairing with the Deligne metric 
(cf., \cite{Del87a}, also \cite{Zha96},\cite{PRS08}), $X_{s}$ is a $\pi$-fiber over $s\in S$. ``${{\it det}(\omega_h)}$" above means 
the natural family of hermitian metrics $\{{\it det}(g_{h_{\bar{s}}})\}_{\bar{s}}$, 
the determinant metric on $K_{\mathcal{X}_{\bar{s}}}$ 
of the K\"ahler metric $g_h$ whose K\"ahler form is 
$c_{1}(\mathcal{L}_{\bar{s}},h_{\bar{s}})$. If $c_1(\mathcal{L}_{\bar{s}},h_{\bar{s}})$ are 
normalised K\"ahler forms of constant scalar curvature K\"ahler metrics, 
then $\bar{\lambda}_{\it CM}$ encodes the Weil-Petersson potential 
essentially by \cite[sections 7,10]{FS90}. This definition is after 
its geometric version ``the CM line bundle'' which is introduced in \cite[section 10,11]{FS90} 
for smooth case and later extended to singular case \cite{PT06}.

\item (Arithmetic Aubin-Mabuchi line bundle) 
We denote the following arithmetic line bundle on $S$ 
as $\bar{\lambda}_{\it AM}(\mathcal{X},\bar{\mathcal{L}})$ 
and call it \textit{arithmetic Aubin-Mabuchi line bundle}:  
$$\bar{\lambda}_{\it AM}(\mathcal{X},\bar{\mathcal{L}})
:=<\bar{\mathcal{L}}^{h},\cdots,\bar{\mathcal{L}}^{h}>,$$ 
where $<\cdots>$ denotes, as above, 
the Deligne pairing with the Deligne metric (cf., \cite{Del87a}, also 
\cite{Zha96},\cite{PRS08}). 
Here, ``AM" stands for 
Aubin-Mabuchi (energy). 

\end{enumerate}

\end{Def}

The above $\bar{\lambda}_{\it CM}$ can be regarded as an extension of 
the modular height $h_K$. Although in the above definitions we suppose smoothness of $\pi$ for 
simplicity. 
We expect that the same construction works when the geometric fibers are 
semi-log-canonical i.e. mildly singular (with regular enough metrics)
\footnote{at least works as a (scheme-theoretic) line bundle 
and the inductive definition of associated metrics 
\cite{Del87a},\cite{Zha96} also at least works at least for 
\textit{almost smooth metrics} in our sense \ref{Not.Cov}}
 and such mildness of singularities and regularity should automatically follow from 
 the Arakelov/arithmetic K-semistability 
as in \cite{Oda13b},\cite{BG14}. Assuming so, then 
our conjecture is roughly as follows (\ref{Ar.K.mod}), which we make more precise in near future. 

If the infimum of $h_{K}$ of all integral models (\textit{model} metrics) is attained by 
a \textit{adelic} vertically semipositive metric model of $(X,L)$, we call 
such ``integral'' model \textit{globally Arakelov K-semistable}. 

\begin{``Conj"}[Arakelov K-moduli]\label{Ar.K.mod}
For each class of (liftable) polarized varieties, 
moduli algebraic stack
$\mathcal{M}$ 
of globally Arakelov K-semistable polarized models 
exists and it has coarse moduli projective scheme $M$ over 
$\mathbb{Z}$. The K-moduli stack (resp., or its coarse moduli scheme) of the class 
of polarized varieties over a fixed field $k$ is $\mathcal{M}\times_{\mathbb{Z}}k$ 
(resp., $M \times_{\mathbb{Z}}k$). 

Furthermore, the arithmetic CM line bundle $\bar{\lambda}_{CM}$ (\ref{a.line} above) on 
$\mathcal{M}$ descends on $M$ as a vertically ample\footnote{cf., e.g., \cite[5.38]{Mor14} for 
the definition} 
arithmetic 
line bundle on $M$ with a natural hermitian metric\footnote{It is obtained as the Deligne metric 
in a modified manner of \ref{aa.line}, applied to the family of cscK metrics 
of the parametrized varieties. More precise meaning will be written in future.}
whose curvature current is the Weil-Petersson current (cf., \cite{FS90}). 
\end{``Conj"}

We wish to call the above arithmetic moduli, \textit{Arakelov K-moduli} as the 
(generalized) Weil-Petersson metric gives a ``canonical Arakelov compactification'' of the arithmetic moduli scheme 
$M$ (cf., also \cite[p77]{Man85}). For references for original geometric versions of the above conjecture, 
please review the (introduction) section \ref{intro} with some explanations. 



\vspace{5mm} \footnotesize \noindent
Contact: {\tt yodaka@math.kyoto-u.ac.jp} \\
Department of Mathematics, Kyoto University, Kyoto 606-8502. JAPAN \\

\end{document}